\documentclass[reqno,11pt]{amsart}
\usepackage[showonlyrefs]{mathtools}
\usepackage{amsmath,amsthm,amssymb,amsfonts,color,xcolor,pifont,cite,a4wide}
\usepackage[colorlinks]{hyperref}
\usepackage{framed,enumerate}
\usepackage[normalem]{ulem}

\newcounter{corr}
\definecolor{violet}{rgb}{0.580,0.,0.827}
\newcommand{\corr}[3]{\typeout{Warning : a correction remains in page
		\thepage}
	\stepcounter{corr}        
	{\color{blue}\ifmmode\text{\,\sout{\ensuremath{#1}}\,}\else\sout{#1}\fi}
	{\color{red}#2}
	{\color{violet} #3}}

\usepackage{pgf,tikz}
\usepackage{mathrsfs,ulem}
\usetikzlibrary{arrows}

\DeclarePairedDelimiter\floor{\lfloor}{\rfloor}

\newcommand{\new}[1]{{\color{red}{#1}}}
\newcommand{\newer}[1]{{\color{blue}{#1}}}
\renewcommand{\newer}[1]{{#1}}
\renewcommand{\new}[1]{{#1}}

\newtheorem{theorem}{Theorem}[section]
\newtheorem{proposition}[theorem]{Proposition}
\newtheorem{corollary}[theorem]{Corollary}
\newtheorem{lemma}[theorem]{Lemma}

\newtheorem{definition}[theorem]{Definition}
\theoremstyle{remark}
\newtheorem{remark}[theorem]{Remark}

\def\LM#1{\hbox{\vrule width.2pt \vbox to#1pt{\vfill \hrule width#1pt height.2pt
}}}
\def\LL{{\mathchoice {\>\LM7\>}{\>\LM7\>}{\,\LM5\,}{\,\LM{3.35}\,}}}
\def\restr{{\LL}}

\newcommand{\R}{\mathbb R}
\newcommand{\N}{\mathbb N}
\newcommand{\Z}{\mathbb Z}

\newcommand{\Om}{\Omega}
\DeclareMathOperator{\Div}{div}
\DeclareMathOperator{\dom}{dom}
\newcommand{\ess}{\textup{ess}\,}
\DeclareMathOperator{\sign}{sign}
\newcommand{\p}{\varphi}
\newcommand{\vp}{\varphi}

\newcommand{\loc}{{\text{loc}}}
\newcommand{\wto}{\rightharpoonup}

\def\dist{\mathrm{dist}}

\def\H{\mathcal H}
\def\eps{\varepsilon}

\author{Antonin Chambolle}
\address{CEREMADE, CNRS and Universit\'e Paris Dauphine, Paris, France}
\email{chambolle@ceremade.dauphine.fr}

\author{Matteo Novaga}
\address{Department of Mathematics, University of Pisa, Pisa, Italy} 
\email{matteo.novaga@unipi.it} 

\numberwithin{equation}{section}

\title{$L^1$--gradient flow of convex functionals}

\begin{document}
\bibliographystyle{plain}

\maketitle

\begin{abstract}
We are interested in the gradient flow of a general first order convex functional with respect to the $L^1$-topology.
By means of an implicit minimization scheme, we show existence of a global limit solution, which satisfies an energy-dissipation estimate, and
solves a non-linear and non-local gradient flow equation, under the assumption of strong convexity of the energy.
Under a monotonicity assumption we can also prove uniqueness of the limit solution,  
even though this remains an open question in full generality. We also consider a geometric evolution corresponding to the 
$L^1$-gradient flow of the anisotropic perimeter. When the initial set is convex, we show that the limit solution 
is monotone for the inclusion, convex and unique until it reaches the Cheeger set of the initial datum. 
Eventually, we show with some examples that uniqueness cannot be expected in general in the geometric case.
\end{abstract}

\tableofcontents

\section{Introduction}
We consider the functional
\[
  \Phi(u):= \int_\Om F(D u) \qquad u\in BV(\Om),
\]
where $\Om$ is a bounded, connected, open subset of $\R^d$,
and $F:\R^d\to [0,+\infty]$ is a convex function 
with \new{$F(\xi)\ge c|\xi|-c'$ for some $c>0$, $c'\ge 0$ and $F(0)=0$}.
\new{Here, $F(Du)$ is understood in the sense of the celebrated paper of
Demengel and Temam~\cite{DM86}:
when the recession function $F^\infty$ (see~\eqref{eq:recession}) of $F$ is not infinite, 
\[
F(Du)=F(\nabla u)\,dx+ F^\infty\left(\frac{D^s u}{|D^s u|}\right)|D^s u|,
\]
where $Du= \nabla u(x)\,dx + D^s u$ is the Radon-Nikod\'ym decomposition
of $Du$ with respect to the Lebesgue measure.
If $F$ is superlinear and hence $F^\infty\equiv +\infty$, then $\Phi(u)<+\infty$
only if $u\in W^{1,1}(\Om)$ and the singular part vanishes in the above formula (see Appendix~\ref{appendixconvex} for details).}

We are interested in the gradient flow of $\Phi$ with respect to the $L^1(\Om)$-topology,
with either homogeneous Neumann, or Dirichlet boundary conditions;
in the latter case the functional
has to be relaxed, with an appropriate boundary integral, if the function $F$ has
linear growth, \new{see Section~\ref{app:Diri}}.

In order to show existence of a \new{gradient flow},
we follow the general approach in~\cite{DG} (see also the comprehensive reference~\cite{AGS}),
which is known as the {\it minimizing movement scheme} and applies to functionals on metric spaces,
under general assumptions.
However, most of the theory developed in~\cite{AGS} does not apply to our setting, since
the Banach space $L^1(\Om)$ does not satisfy the Radon-Nikod\'ym property (see~\cite[Remark~1.4.6]{AGS}).
In particular, we cannot derive uniqueness of gradient flow solution from general results, 
and we are able to prove it only in some special cases.

For this reason, the are few results in the literature concerning $L^1$-gradient flows. 
In~\cite{Dayrens} the author considers the $L^1$-gradient flow of a second order functional related to the Willmore energy,
and studies in detail rotationally symmetric solutions.
We also mention~\cite{RossiStefanelliThomas} where the authors, motivated by a model
of delamination between elastic bodies, study a monotone geometric flow by means of a minimizing movement scheme 
reminiscent to the one in Section~\ref{sec:per}. They show existence of a limit solution and discuss some examples.

\new{We recall that the De~Giorgi minimizing movements scheme for building gradient flows in a metric space~\cite{AGS} typically builds $u^{n+1}\approx u((n+1)\tau)$ as a minimizer of
\[
\min_u \Phi(u) + \frac{1}{2\tau}\dist(u,u^n)^2.
\]
The exponent $2$ here is crucial, as it ensures formally that $\dist(u^{n+1},u^n)\approx \tau|D\Phi(u^{n+1})|$ (for an appropriate definition of the latter expression), as expected in an (implicit) Euler scheme. In particular, a motion always occur if the initial point $u^0$ is not critical for $\Phi$, contrarily to
what would happen, in the present paper, if we used an exponent $1$ and the
distance induced by the $L^1$ norm (then as soon as $\Phi$
 has a subgradient at $u^0$ which is bounded, one easily obtains that no motion occurs for $\tau$ small enough).
}
\smallskip

The plan of the paper is the following: in Section~\ref{sec:existence} we introduce the minimizing movements and 
we show convergence of the discrete solutions to a limit solution. We also show a general dissipation 
estimate from which we derive, under the assumption of strong convexity of $F$, a gradient flow equation satisfied by the limit solution.

In Section~\ref{sec:monotone} we analyze the case when the initial datum is a subsolution (see Definition~\ref{defsub})\new{,
  in the case where $F$ is strictly convex with superlinear growth}. 
In such case the limit solution is non-decreasing  in time and it is indeed unique.

In Section~\ref{sec:dirichlet} we consider the simplest possible functional, that is, the Dirichlet energy. In this particular case we can show
a stronger uniqueness result, namely that the limit gradient flow equation always admits a unique solution.

Finally, in Section~\ref{sec:per} we consider the geometric evolution corresponding to the $L^1$-gradient flow of the anisotropic perimeter.
Even if we are not able to characterize the limit flow as we do in the case of functions, when the initial set is convex, we can prove that 
the evolution is unique, monotone for the inclusion, and remains convex until it reaches the Cheeger set of the initial set. 
In two dimensions we also show that it stays convex until it becomes a Wulff Shape, and then shrinks to a point in finite time. 
Simple examples show that the geometric evolution is in general non-unique, after reaching the Cheeger set. \new{Appendix~\ref{appendixconvex} discusses the definition
and main properties of the functional $\Phi$.}

\smallskip
{\it Acknowledgements.} The second author is member of INDAM-GNAMPA and was supported by the PRIN Project 2019/24. Part of this work was done while he visited CEREMADE, supported by Univ.~Paris-Dauphine PSL.

\section{Existence of solutions}\label{sec:existence}

\subsection{Minimizing movements}

Following \cite{DG}, we introduce the $L^1$-minimizing movement scheme. Given $u^0\in L^1(\Om)$,
we let $u^n$, for $n\ge 1$, be a minimizer of
\begin{equation}\label{eq:MMPhi}
  \min_{\new{u\in L^1(\Om)}} \Phi(u) + \frac{1}{2\tau} \left(\int_\Om |u-u^{n-1}|dx\right)^2.
\end{equation}

If $F$ has superlinear growth, then $u^n\in W^{1,1}(\Om)$. Assuming
in addition that $F$ is strictly convex, we deduce that if $u'$ is
another solution, $Du'=Du^n$ a.e., and $u'-u^n$ is a constant.
As a consequence, any other solution is of the form $u^n+c$
where $c$ is a minimizer of $\|u^n-u^{n-1}-c\|_1$, that is, a median
value of $u^n-u^{n-1}$. Notice that, by convexity,
the set of median values is an interval.
If $u^{n-1}\in W^{1,1}(\Om)$
(which is true for $n\ge 2$, and which will we assume for $n=1$),
then, since $u^n-u^{n-1}\in W^{1,1}(\Om)$ and $\Om$ is connected,
it has a unique median value, hence we have the following result.

\begin{lemma}
  Assume that $F$ is stricly convex with superlinear growth,
  and that $\Phi(u^0)<+\infty$.
  Then for any $n\ge 1$,
  there is a unique minimizer to~\eqref{eq:MMPhi}.
\end{lemma}

\begin{remark}\label{rem:surfmin}
In case $F$ is not strictly convex or the growth is not superlinear,
the uniqueness is not guaranteed. However, in that case,
\begin{enumerate}
\item \label{rem-l1u} by strong convexity in $u\mapsto \|u-u^{n-1}\|_1$ of the
  energy, one easily sees that given any two minimizers $u,u'$
  of~\eqref{eq:MMPhi}, $\|u-u^{n-1}\|_1=\|u'-u^{n-1}\|_1$\new{. Indeed,
    one has, for $\theta\in (0,1)$,
    \begin{multline*}
      \|\theta u+(1-\theta) u'-u^{n-1}\|_1^2
      \le \left( \theta \|u-u^{n-1}\|_1+(1-\theta)\|u'-u^{n-1}\|_1\right)^2
      \\ \le \theta \|u-u^{n-1}\|_1^2
      +(1-\theta)  \|u'-u^{n-1}\|_1^2
      -\theta(1-\theta) \left(\|u-u^{n-1} \|_1-\|u'-u^{n-1}\|_1 \right)^2,
    \end{multline*}
    showing that $(u+u')/2$ would be otherwise a better minimizer};
\item \label{rem-ms} one can easily build measurable selections of the
  solutions $\tau\mapsto u^n$ as $\tau$ varies, as follows.
  A first observation is that for any $p \in[1,d/(d-1)]$, if the
  energy of $u$ in~\eqref{eq:MMPhi} is finite, then $u\in L^p(\Om)$,
  by Sobolev's embedding and using that $\Phi(u)$ controls the
  total variation of $u$. Then, given $p\in (1,d/d-1)$,
  for $\eps>0$, one can consider
  the unique minimizer $u^\eps_\tau$ of the strictly convex energy:
  \[
    \Phi(u) + \frac{1}{2\tau} \left(\int_\Om |u-u^{n-1}|dx\right)^2
    +\eps\int_\Om |u|^p dx
  \]
  and one easily shows that $\tau\mapsto  u^\eps_\tau$ is continuous
  (in $L^1(\Om)$, as well as $L^p(\Om)$).
  Sending $\eps\to 0$ we find that $u^\eps_\tau\to u_\tau$, the solution
  of~\eqref{eq:MMPhi} with minimal $L^p$ norm. 
  \new{Indeed, if $u$ is another solution, one can write
  \begin{align*}
  \Phi(u^\eps_\tau)+ \frac{1}{2\tau} &\left(\int_\Om |u^\eps_\tau-u^{n-1}|dx\right)^2
    +\eps\int_\Om |u^\eps_\tau|^p dx \\ &
    \le
    \Phi(u)+ \frac{1}{2\tau} \left(\int_\Om |u-u^{n-1}|dx\right)^2
    +\eps\int_\Om |u|^p dx \\ &
    \le \Phi(u^\eps_\tau)+ \frac{1}{2\tau} \left(\int_\Om |u^\eps_\tau-u^{n-1}|dx\right)^2
    +\eps\int_\Om |u|^p dx.
  \end{align*}
  Hence, $\int_\Om |u^\eps_\tau|^p dx \le \int_\Om |u|^p dx$ and we conclude
  thanks to the lower-semicontinuity of the $p$-norm.
  The limit} $u_\tau$ is thus a \new{(Bochner)-}measurable
  selection. \new{We also obtain~\cite[Thm~8.28]{Leoni2} that
  $(\tau,x)\mapsto u_\tau(x)$ is measurable.}
\end{enumerate}
\end{remark}

We can now define $u_\tau(t) := u^{\floor{t/\tau}}$ where $\floor{\cdot}$ is
the integer part, and we show the following theorem (whose proof is classical).
\begin{theorem}\label{th:minMM}
  Assume that $\Phi(u^0)<+\infty$. Then, there exists
  $u\in C^0([0,+\infty); L^1(\Om))$ and a subsequence $\tau_k\to 0$
  such that $u_{\tau_k}\to u$ in $L^\infty([0,T];L^1(\Om))$, for all $T>0$,
  and
  \[
    \|u(s)-u(t)\|_1 \le \sqrt{2\Phi(u^0)}\sqrt{|t-s|}
\]
for any $t,s\in [0,T]$.
\end{theorem}

\begin{remark} If we consider the piecewise affine interpolant $\hat{u}_\tau$ of $u^n$ in time,
defined as $u^n + (t/\tau-1)(u^{n+1}-u^n)$ for $n\tau\le t\le (n+1)\tau$,
  rather than the piecewise constant interpolant, then the convergence is also in $C^0([0,T];L^1(\Om))$.
\end{remark}

\begin{proof}
For any $0\le m <n$, we have
\[
\begin{aligned}
 \|u_\tau(n\tau)-u_\tau(m\tau)\|_1^2 &\le 
 \left(\sum_{k=m}^{n-1}\|u_\tau((k+1)\tau)-u_\tau(k\tau)\|_1\right)^2
 \\
 &\le (n-m)\sum_{k=m}^{n-1}\|u_\tau((k+1)\tau)-u_\tau(k\tau)\|_1^2
  \\
 &\le 2\tau(n-m) \sum_{k=m}^{n-1}\left( \Phi(u_\tau(k\tau)-\Phi(u_\tau((k+1)\tau)\right) 
 \\
 &= 2 (\Phi(u^m)-\Phi(u^n))  (n\tau -m\tau) 
\\ & \le 2 \Phi(u^0)  (n\tau -m\tau),
\end{aligned}
\]
where we used the Cauchy-Schwarz inequality and the minimality of $u_\tau(k\tau)$,
and the fact that the sequence $(\Phi(u^n))_{n}$ is non-increasing.

We deduce in addition that $\Phi(u_\tau(t))\le \Phi(u^0)$ for any $t>0$, so that,
thanks to the assumptions on $F$ and
together with the bound on $\|u_\tau (t)-u^0\|_1$, we find that there is a compact
subset of $L^1(\Om)$ (even $L^p(\Om)$, for $p<d/(d-1)$) which contains $u_\tau(t)$
for any $t>0$.

For any $t,s\ge 0$, if follows that
\[
  \|u_\tau(t)-u_\tau(s)\|_1 \le \sqrt{2\Phi(u^0)}\sqrt{\left|\floor{t/\tau}-\floor{s/\tau}\right|\tau}
  \le \sqrt{2\Phi(u^0)}\sqrt{\tau+|t-s|}.
\]

The compactness and convergence is then deduced by the Ascoli-Arzel\`a Theorem.
\end{proof}

By a simple interpolation argument, we can
show a slightly improved convergence for the previous theorem.

\begin{proposition}\label{prop:interpol}
  Let $p\in [1,d/(d-1))$. Then the subsequence $(u_{\tau_k})_k$
  in Theorem~\ref{th:minMM}
  also converges to $u$ in $L^\infty([0,T];L^p(\Om))$ for any
  $T>0$, while the piecewise-affine interpolants $\hat{u}_{\tau_k}$
  converge in $C^0([0,T];L^p(\Om))$.
\end{proposition}

\begin{proof}
  By construction, for $t\in [0,T]$ the norms $\|u_\tau(t)\|_{d/(d-1)}$ are
  uniformly bounded and for $1<p<d/(d-1)$, there is a compact
  set of $L^p(\Om)$ such that $u_\tau(t)\in C_p$.
  
For $0<\epsilon<1$, writing $|u_\tau(t)-u_\tau(s)|^p=|u_\tau(t)-u_\tau(s)|^{1-\epsilon}|u_\tau(t)-u_\tau(s)|^{p-1+\epsilon}$ and using H\"older's inequality,
  we have
  \[
    \|u_\tau(t)-u_\tau(s)\|_p^p \le \|u_\tau(t)-u_\tau(s)\|_1^{1-\epsilon}
    \left(    \int_\Om |u_\tau(t)-u_\tau(s)|^{\frac{p-1+\epsilon}{\epsilon}}\right)^\epsilon
  \]
  hence if $(p-1)/\epsilon + 1\le d/(d-1)$, for
  instance for $\epsilon=(p-1)(d-1)<1$ (or any $\epsilon<1$ if $d=1$),
  we find that
  \[
    \|u_\tau(t)-u_\tau(s)\|_p \le C\sqrt{\tau + |t-s|}^{\frac{d}{p}-(d-1)}
  \]
  Hence, the convergence is also in $L^\infty([0,T],L^p(\Om))$.
\end{proof}

\subsection{Euler-Lagrange equation}

The Euler-Lagrange equation for $u^n$ minimizing~\eqref{eq:MMPhi} 
takes \new{formally}
the form:
\begin{equation}\label{eq:ELgeneral}
  \begin{cases}
    -\Div z^n + \frac{\|u^n-u^{n-1}\|_1}{\tau} \sign(u^n-u^{n-1})\ni 0  \\[2mm]
    z^n\cdot Du^n = F(Du^n) + F^*(z^n)
  \end{cases}
\end{equation}
\new{(with $F^*$ the convex conjugate of $F$, see Appendix~\ref{appendixconvex})},
where the last statement \new{should be in the sense of~\cite{Anzellotti}}
if $F$ has minimal growth $1$ ($Du$
can be a measure), otherwise we just \new{expect} $z^n\in\partial F(Du^n)$~\new{a.e}.

This follows from~\cite[Prop.~5.6]{EkelandTemam},
applied in $V=L^1(\Om)$ and $V^*=L^\infty(\Om)$. In that case,
$u\mapsto \|u-u^{n-1}\|^2/(2\tau)$ is everywhere continuous
while $\Phi$ is lower semicontinuous. Hence,
$\partial (\Phi(\cdot)+ \|\cdot-u^{n-1}\|_1^2/(2\tau))
= \partial\Phi + \partial\|\cdot-u^{n-1}\|_1^2/(2\tau)$, where
the subgradients are elements of $L^\infty(\Om)$. So a minimizer $(u^n)$ is
characterized by
\begin{equation}\label{eq:ELcompactform}
  0 \in
  \partial\Phi(u^n)  + \frac{\|u^{n}-u^{n-1}\|_1}{\tau}\sign(u^n-u^{n-1})
\end{equation}
where $\sign(t) = \{1\}$ for $t>0$, $\{-1\}$ for $t<0$, and $[-1,1]$ for $t=0$.

\new{Then, in case $F$ is $1$-homogeneous, \eqref{eq:ELgeneral} is deduced
from~\cite[Prop.~3]{MollAnisotropic} (in that case, the second
equation in~\eqref{eq:ELgeneral} is to be understood in the sense of~\cite{Anzellotti}). The more general Lipschitz case is studied
in~\cite{gorny2022dualitybased}. In case both $F$ and $F^*$ are
superlinear, Lemma~\ref{lem:subgrad} (or Lemma~\ref{lem:subgradDiri}) in Appendix~\ref{appendixconvex} also shows~\eqref{eq:ELgeneral}.
A general case (e.g., $F$ neither Lipschitz nor superlinear) remains unclear. 
}

\new{
\begin{remark}\label{rmk:qmeasurable}
For varying $\tau>0$, let us denote $u_\tau$ the minimizer of~\eqref{eq:MMPhi} and
$q_\tau\in\partial\Phi(u_\tau)$ the corresponding subgradient in the Euler-Lagrange
equation~\eqref{eq:ELcompactform}. Then, as in Remark~\ref{rem:surfmin}-(\ref{rem-ms}), one can build a measurable selection of $\tau\mapsto q_\tau$. One first observes
that $q_\tau$ minimizes the dual problem (with $\Phi^*$ the convex
conjugate of $\Phi$)
\[
\min_q \Phi^*(q) -\int_\Om q(x)u^{n-1}(x)\,dx + \frac{\tau}{2}\|q\|_\infty^2.
\]
This is easily deduced from~\eqref{eq:ELcompactform} and the fact
$u_\tau\in\partial\Phi^*(q_\tau)$.
Then, one perturbs this problem by adding a term $\eps\|q\|^2_{p'}/2$ for some
$p'\in (1,+\infty)$ (for instance, $p'=2$ --- using $p'\ge d$ is less crucial
as requiring $p\le d/(d-1)$ in the primal problem, since $q$ has to be bounded anyway).
This allows to define a unique minimizer $q^\eps_\tau$, which in addition
is continuous with respect to $\tau$ in $L^{p'}(\Om)$.

For each $\tau$, as $\eps\to 0$, this minimizer $q^\eps_\tau$ goes to the solution 
$q_\tau$ of the dual problem which is minimal in $L^{p'}$-norm, and is thus
a Bochner-measurable selection (and measurable as a function of $(\tau,x)$).
\end{remark}
}
\subsection{Estimate of the time derivative}

The class of functionals $\Phi$ we are considering satisfies the
following fundamental estimate: for any $u,v\in L^1(\Om)$,
\begin{equation}\label{eq:submod}
  \Phi(u\wedge v) + \Phi(u\vee v) \le \Phi(u)+\Phi(v)
\end{equation}
(with equality if $F$ has superlinear growth)\new{, see Lemma~\ref{lem:submod}}. Here
for $x,y\in\R$,
$x\vee y=\max\{x,y\}$ and $x\wedge y=\min\{x,y\}$ and the notation
extends to real-valued functions.
In this context, we can prove the following:
\begin{lemma}\label{lem:control}
  Let $v\in L^1(\Om)$, $q\in L^\infty(\Om)$ with $q\in\partial \Phi(v)$.
  Let $u$ be a minimizer of:
  \[
    \Phi(u) + \frac{1}{2\tau}\|u-v\|_1^2.
  \]
  Then
  \[
    \frac{\|u-v\|_1}{\tau}\le \|q\|_\infty.
  \]
\end{lemma}
\begin{proof}
  The following remark is crucial: if $q\in\partial\Phi(v)$, $p\in\partial\Phi(u)$, then (denoting as usual $x^+=x\vee 0$ and $x^- = (-x)^+$):
  \begin{equation}\label{eq:subplus}
    \int_\Om (q-p)(v-u)^+ dx \ge 0 \quad\textup{ and }\quad
    \int_\Om (q-p)(v-u)^- dx \le 0.
  \end{equation}
  Indeed, one has:
  \begin{equation}\label{eq:vcompar}
    \Phi(u\vee v)\ge \Phi(u) + \int_\Om p(u\vee v-u)dx
    \quad\textup{ and }\quad
    \Phi(u\wedge v)\ge \Phi(v) + \int_\Om q(u\wedge v-v)dx.
  \end{equation}
  Using that $u\vee v-u=(v-u)^+$ and $u\wedge v-v=-(v-u)^+$,
  the first inequality in~\eqref{eq:subplus}
  follows by summing the two previous inequalities and using~\eqref{eq:submod}.
  The second is proved similarly.

  Since in the Lemma, $u$ satisfies the equation
  (\textit{cf}~\eqref{eq:ELcompactform}):
  \[
    \exists p\in  \partial\Phi(u) \cap  -\sign(u-v)\frac{\|u-v\|_1}{\tau},
  \]
  we deduce from~\eqref{eq:subplus} that (here ``sign'' is single-valued
  as the integrand vanishes for $v\le u$):
  \begin{equation}\label{eq:halfcontrol}
    0\le    \int_\Om \left(q-\sign(v-u)\frac{\|u-v\|_1}{\tau}\right)(v-u)^+dx
    \le \left((\ess\sup_\Om q)-\frac{\|u-v\|_1}{\tau}\right)\int_\Om (v-u)^+dx
  \end{equation}
  so that if $\{v>u\}$ has positive measure,
  $\frac{\|u-v\|_1}{\tau}\le \ess\sup_\Om q$.
  Similarly (multiplying with $-(v-u)^-$) we show that
  so that if $\{v<u\}$ has positive measure,
  $\frac{\|u-v\|_1}{\tau}\le -\ess\inf_\Om q$.
  The thesis follows.
\end{proof}
 We deduce immediately the following result, as a consequence of
 Lemma~\ref{lem:control} and the Euler-Lagrange equation~\eqref{eq:ELcompactform}.
 \begin{theorem}\label{th:controlspeed}
   Let $u^0\in L^1(\Om)$, $\tau>0$ and $(u^n)_{n\ge 0}$ defined
   by the minimizing movement scheme. Then \begin{itemize}
   \item[i.] for any $n\ge 1$,
     \[
       \frac{\|u^{n+1}-u^n\|_1}{\tau} \le \frac{\|u^{n}-u^{n-1}\|_1}{\tau}\ ;
     \]
     \newer{
     \item[\new{ii.}] as a result,
       \[
         \frac{\|u^{n+1}-u^n\|_1}{\tau} \le  \sqrt{\frac{2\Phi(u^0)}{(n+1)\tau}}\ ;
       \]
     }
   \item[\new{iii.}] if in addition $\partial \Phi(u^0)\neq\emptyset$, then
     for any $n\ge 0$,
     \[
       \frac{\|u^{n+1}-u^n\|_1}{\tau} \le \|\partial^0\Phi(u^0)\|_\infty
     \]
     where $\partial^0\Phi$ denotes the element in the subgradient with minimal norm.
   \end{itemize}
 \end{theorem}
 \new{We observe that the set of $u^0$ such that $\partial\Phi(u^0)$ contains
   a bounded element is dense in the domain of $\Phi$, see Lemma~\ref{lem:density}.}
 \newer{\begin{proof}
     Only point (ii.)~still needs to be proven. We write for $n\ge 0$ (thanks to point (i.)):
     \[
       \frac{\|u^{n+1}-u^n\|^2_1}{\tau^2}\le \frac{2}{(n+1)\tau}\sum_{k=0}^n \frac{\|u^{k+1}-u^k\|_1^2}{2\tau} \le \frac{2}{(n+1)\tau}\sum_{k=0}^n (\Phi(u^k)-\Phi(u^{k+1})),
     \]
     and the claim follows.         
 \end{proof}}
 \begin{corollary} \label{cor:Lip}
   Let $u$ be an evolution provided by Theorem~\ref{th:minMM}.
   Then $u$ is \newer{locally} Lipschitz in time. \new{Its time derivative is a bounded
     measure of the form $\dot{u}(t)\otimes dt$ which satisfies,} for a.e.~$t\ge 0$,
   \newer{
   \begin{equation}\label{eq:boundtimegeneral}
     |\dot{u}(t)|(\Omega)\le \sqrt{\frac{2\Phi(u^0)}{t}}.
   \end{equation}
   If in addition  $\partial\Phi(u^0)\neq\emptyset$, then}
 \begin{equation}\label{eq:boundtime}
   \new{|\dot{u}(t)|(\Om)} \le \|\partial^0\Phi(u^0)\|_\infty.
 \end{equation}
 \end{corollary}
 
   \new{\begin{proof}
       Indeed, we observe first that given $\eta$ a smooth function with compact support
       in $ (t_1,t_2)\times\Om$, $t_1<t_2$, one has for a.e.~$t>0$:
       \begin{align*}
       \int_{t_1}^{t_2}\int_\Om u(t,x)\partial_t\eta(t,x)dt dx
       &= \lim_{\tau\to 0} 
       \int_{t_1}^{t_2}\int_\Om u(t,x)\frac{\eta(t+\tau,x)-\eta(t,x)}{\tau} dt dx\\\
       &= \lim_{\tau\to 0} 
       \int_{t_1}^{t_2}\int_\Om \frac{u(t-\tau,x)-u(t,x)}{\tau}\eta(t,x) dt dx
       \\ & \le \min\left\{\|\partial^0\Phi(u^0)\|_\infty,\sqrt{\frac{2\Phi(u^0)}{t_1}}\right\} \int_{t_1}^{t_2}\|\eta(t,\cdot)\|_\infty dt,
       \end{align*}
       thanks to Theorem~\ref{th:controlspeed}.
We deduce that $\dot u$ is a measure, whose marginals are in addition absolutely continuous with respect to the Lebesgue measure $dt$. Hence one can 
       disintegrate $\dot u$ as $\dot u(t)\otimes dt$, and it follows that
       for a.e.~$t$, \eqref{eq:boundtimegeneral}-\eqref{eq:boundtime} hold.
   \end{proof}}
   \newer{
   \begin{corollary} \label{cor:LipPhi}
     Let $u$ be an evolution provided by Theorem~\ref{th:minMM}, assuming as always $\Phi(u^0)<+\infty$.
     Then $t\mapsto \Phi(u(t))$ is locally Lipschitz. More precisely,
     \begin{itemize}
     \item[i.] If $\partial\Phi(u^0)\neq\emptyset$, then $|\Phi(u(s)-\Phi(u(t)|\le \|\partial^0\Phi(u^0)\|_\infty^2 |s-t|$ for any $s,t\ge 0$;
     \item[ii.] In general, for $s>t>0$,
       \[
         |\Phi(u(s))-\Phi(u(t))|\le \frac{2\Phi(u^0)}{t}|s-t|.
       \]
     \end{itemize}
     In particular, $d\Phi(u(t))/dt$  exists for almost every $t>0$.
   \end{corollary}
   \begin{proof}
     For every $v\in L^1(\Om)$ and $n\ge 1$, one can write thanks to~\eqref{eq:ELcompactform}:
     \[
       \Phi(v)\ge \Phi(u^n) - \frac{\|u^n-u^{n-1}\|_\tau}{\tau}\int_\Om
       \sign(u^n-u^{n-1})(v-u^n)dx \ge
       \Phi(u^n) - \frac{\|u^n-u^{n-1}\|_1}{\tau}\|v-u^n\|_1.
     \]
     Using Theorem~\ref{th:controlspeed} (iii.) and letting $\tau\to 0$ with
     $\tau n\to t$, we find:
     \[
       \Phi(v) \ge \Phi(u(t)) - \|\partial^0\Phi(u^0)\|_\infty \|v-u(t)\|_1
     \]
     for all $t\ge 0$.
     Then we conclude choosing $v=u(s)$, and observing that Corollary~\ref{cor:Lip} yields that $\|u(t)-u(s)\|_1\le
     \|\partial^0\Phi(u^0)\|_\infty|t-s|$.

     Alternatively, we can also bound  $\frac{\|u^n-u^{n-1}\|_1}{\tau}\|v-u^n\|_1$
     using Theorem~\ref{th:controlspeed} (ii.), and we obtain,
     letting again $\tau\to 0$ with $n\tau\to t$:
     \[
       \Phi(u(t))\le\Phi(v) + \sqrt{\frac{2\Phi(u^0)}{t}}\|v-u(t)\|_1.
     \]
     We then choose $v=u(s)$ (for $s<t$ and $s>t$), 
     and the thesis follows from Corollary~\ref{cor:Lip}.
   \end{proof}
   
 }
 \subsection{Dissipation estimate}
 We shall prove the following dissipation estimate. \new{This is a variant of~\cite[Thm.~2.3.3]{AGS}, yet our time derivative is here a measure
 while we still wish to consider the slopes as elements in $L^\infty(\Omega)$.}
 \begin{theorem}\label{th:dissip}
   Let $u^0$ satisfy $\Phi(u^0)<+\infty$ and let $u$ be a limit of minimizing movements
   given by Theorem~\ref{th:minMM}. Then, for any $t> 0$,
   $\dot{u}$ is a measure with marginal $s\mapsto |\dot u(s)|(\Om)$ in $L^2(0,t)$
   and there exists $q\in L^2((0,t);L^\infty(\Om))$ with $q(s)\in -\partial\Phi(u(s))$
   for a.e.~$s\ge 0$ such that
\begin{equation}\label{eq:dissip}
  \Phi(u(t)) + \frac{1}{2} \int_0^t (|\dot{u}(s)|(\Om))^2 ds
  + \frac{1}{2}\int_0^t \|q(s)\|^2_\infty ds\le \Phi(u^0).
\end{equation}  
 \end{theorem}
 
 \begin{proof}
We remain in the framework of Theorem~\ref{th:minMM}, assuming
that $\Phi(u^0)<+\infty$ and that $u_\tau$, defined above converges, up to a subsequence,
to a function $u\in C^{0,1/2}([0,T];L^1(\Om))$.

As usual (see for instance~\cite[Sec.~3.2]{AGS}), for $n\tau < t <(n+1)\tau$, we let $\tilde{u}_\tau(t)$
be a minimizer of
\[
  \min_u \Phi(u) + \frac{1}{2(t-n\tau)} \|u-u^n\|^2_1,
\]
which satisfies the Euler-Lagrange equation
\begin{equation}\label{eq:elDG}
  \partial \Phi(\tilde{u}_\tau(t))
  + \frac{\|\tilde{u}_\tau(t)-u^n\|_1}{2(t-n\tau)} \sign(\tilde{u}_\tau(t)-u^n)\ni 0.
\end{equation}
By Remark~\ref{rem:surfmin}-\eqref{rem-l1u}, observe that even if the minimizer might be
non-unique, the value of $\|\tilde{u}_\tau(t)-u^n\|_1$ is. In any
case, as mentioned in Remark~\ref{rem:surfmin}-(\ref{rem-ms}), we assume
that $t\mapsto \tilde u_\tau(t)$ is measurable.
We also have
that $\|\tilde{u}_\tau(t)-u_\tau(t)\|_1 \le \sqrt{2\Phi(u^0)(t-\new{n}\tau)} \new{\le \sqrt{2\Phi(u^0)\tau}}$, so that
$\tilde{u}_\tau$ converges to the same limit
as $u_\tau$, also uniformly in time.

Now, for $n\ge 0$, $0<s<\tau$, we let
$h(s) =
\Phi(\tilde{u}_\tau(n\tau + s))+
\|\tilde{u}_\tau(n\tau + s)-u^n\|^2_1/(2s)$, hence
$h(\tau)=\Phi(u^{n+1})+\|u^{n+1}-u^n\|_1^2/(2\tau)$ and $\lim_{s\to 0} h(s)=\Phi(u^n)$.
It is standard that:
\[
  h'(s) \le -\frac{\|\tilde{u}_\tau(n\tau + s)-u^n\|^2_1}{2s^2}, 
\]
so that (using $h(\tau)=\lim_{\epsilon\to 0} h(\epsilon)+\int_\epsilon^\tau h'(s)ds$)
\[
  \Phi(u^{n+1})+\frac{\|u^{n+1}-u^n\|_1^2}{2\tau}
  \le \Phi(u^n) -\frac{1}{2} \int_0^\tau \frac{\|\tilde{u}_\tau (n\tau + s)-u^n\|_1^2}{s^2} ds.
\]
Thanks to the Euler-Lagrange equation~\eqref{eq:elDG}, we deduce:
\[
  \Phi(u^{n+1})  +
  \frac{1}{2}\int_{n\tau}^{(n+1)\tau} \|\dot{\hat{u}}_\tau(s)\|_1^2ds
  +\frac{1}{2}\int_{n\tau}^{(n+1)\tau} \|q_\tau(s)\|^2_\infty ds
  \le\Phi(u^n),
\]
where for all $t$, $q_\tau(t)\in -\partial\Phi(\tilde{u}_\tau(t))$
\new{(and we also assume, reasoning as in Remark~\ref{rmk:qmeasurable}, that
$q_\tau$ is measurable),}
and $\hat u(t)$ is the piecewise-affine interpolant,
 which also converges to $u$ up to a subsequence
(in $C^0([0,T];L^p(\Om))$ for $1\le p \le d/(d-1)$, see Prop.~\ref{prop:interpol}). Summing this inequality from $n=0$
to $\floor{t/\tau}-1$, for $0<t\le T$, we find:
\[
  \Phi(u_\tau(t))+\frac{1}{2}\int_0^{t-\tau}\|\dot{\hat{u}}_\tau(s)\|_1^2ds
  +\frac{1}{2}\int_{0}^{t-\tau} \|q_\tau(s)\|^2_\infty ds
  \le\Phi(u^0).
\]
By lower-semicontinuity of the convex functions appearing in the
integrals 
we claim that~\eqref{eq:dissip} is deduced,
where $q$ is a weak limit (in $L^2([0,T];L^{p'}(\Om))$) of $q_\tau$,
and $p'$ the conjugate exponent of some $p\in (1,d/(d-1))$.

The only difficulty is with the measure term. Given $\varphi \in
C_c^\infty([0,T)\times \Om)$, it is not difficult to check that
for $\tau$ small enough:
\[
  \frac{1}{2}\int_0^{t-\tau}\|\dot{\hat{u}}_\tau(s)\|_1^2ds
  \ge \int_\Om \varphi(0,x)u^0(x)dx -\int_0^t \int_\Om \dot{\varphi} \hat{u}_\tau dx ds
  - \frac{1}{2}\int_0^t \|\varphi(s)\|_{\infty}^2ds
\]
so that, passing to the limit along an appropriate subsequence,
\[
  \int_\Om\varphi(0)u^0dx- \int_0^t \int_\Om \dot{\varphi} u dx ds
   - \frac{1}{2}\int_0^t \|\varphi(s)\|_{\infty}^2ds
   \le
   \liminf_{\tau\to 0}
     \frac{1}{2}\int_0^{t-\tau}\|\dot{\hat{u}}_\tau(s)\|_1^2ds=:\ell.
\]
In particular (using also that $u(t)\to u^0$ as $t\to 0$),
we deduce immediately that the distribution $\dot{u}$ is
a bounded Radon measure (in $[0,T)\times \Om$), satisfying for all $t\le T$:
\[
  \int_{[0,t]\times \Om} \varphi d\dot u
     - \frac{1}{2}\int_0^t \|\varphi(s)\|_{\infty}^2ds
   \le \ell.
 \]
 Letting $n\ge 1$, $0=t_0<t_1<\cdots< t_n=t$ and considering $m_i\ge 0$, $i=1,\dots,n$,
 and the supremum over all functions $\varphi$ with $\varphi_{|(t_{i-1},t_i)}\in C_c^\infty([t_{i-1},t_i)\times\Om)$ \new{and} $\|\varphi\|_{L^\infty(t_{i-1},t_i)}\le m_i$ we deduce:
 \[
   \sum_{i=1}^n m_i|\dot u|([t_{i-1},t_i)\times\Om) - (t_{i}-t_{i-1})\frac{m_i^2}{2}\le\ell.
 \]
 By uniform approximation of a smooth function $\psi\in C_c^\infty([0,t);\R_+)$ by piecewise
 constant functions, we deduce that the marginal measure $|\dot u|(\Om)$ in $(0,t)$ satisfies:
 \[
   \int_0^t\psi(s)d(|\dot u|(\Om))(s) - \frac{1}{2}\psi^2(s) ds \le \ell
 \]
 and it follows that $|\dot u|(\Om)$ is indeed in $L^2(0,t)$, with
 \[
   \frac{1}{2}\int_0^t (|\dot u|(\Om))^2 ds \le \ell.
 \]

Now, we check that $q(t)\in \partial\Phi(u(t))$ a.e.:
given $\varphi\in C_c^\infty((0,T)\times \Om)$, we have
\[
  \int_0^T \Phi(\varphi(t))dt
  \ge \int_0^T \Phi(\tilde{u}_\tau(t))dt +
  \int_0^T \int_\Om q_\tau(t,x)(\tilde{u}_\tau(t,x)-\varphi(t,x)) dx dt.
\]
Since $\tilde{u}_\tau\to u$ in $L^\infty([0,T];L^p(\Om))$
(using Prop.~\ref{prop:interpol}) and 
$q_\tau\rightharpoonup q$ in $L^2([0,T];L^{p'}(\Om))$, we obtain that
\[
  \int_0^T \int_\Om q_\tau(t,x)\tilde{u}_\tau(t,x) dx dt
  \to \int_0^T\int_\Om q(t,x)u(t,x)dx dt.
\]
It follows that
\[
  \int_0^T \Phi(\varphi(t))dt
  \ge \int_0^T \Phi(u(t))dt +
  \int_0^T \int_\Om q(t,x)({u}(t,x)-\varphi(t,x)) dx dt.  
\]
We deduce that for a.e.~$t$, $-q(t)\in \partial\Phi(u(t))$.
\end{proof}

A dissipation estimate like \eqref{eq:dissip} usually implies
that the flow $u(t)$ is a curve of maximal slope in the sense of~\cite[Def.~1.3.2]{AGS},
satisfying 
\begin{equation}\label{eqslope}
\new{\frac{d\Phi(u(t))}{dt}} = -\int_\Om q(t)\new{\dot u(t)}\,dx\qquad  \text{for a.e.~$t\ge 0$.}
\end{equation}
However, as already observed in the Introduction, the results in \cite{AGS} fail to apply in the $(1,\infty)$-duality,
since $L^1(\Om)$ does not satisfy the Radon-Nikod\'ym property\new{, and it is not obvious to give a meaning to~\eqref{eqslope} in this context}.

We shall rigorously prove \eqref{eqslope} in the next section,
under the additional assumption that $F$ is 
\new{strongly convex}.



\subsection{Strongly convex case}\label{sec:sc}
 In this part, we first assume that in addition there exists $\gamma>0$ such that $F$ is $\gamma$-convex:
 \[
   F(\eta)\ge F(\xi) + p\cdot(\eta-\xi) +\frac{\gamma}{2}|\eta-\xi|^2
 \]
 for any $\eta,\xi\in\R^d$ and $p\in\partial F(\xi)$.
 Then~\eqref{eq:vcompar} becomes\new{, still given $q\in\partial\Phi(v)$, $p\in\partial\Phi(u)$ (and in particular $u,v\in\dom\Phi\subseteq H^1(\Om)$):}
\begin{align*}
&  \Phi(u\vee v)\ge \Phi(u) + \int_\Om p(u\vee v-u)dx+\frac{\gamma}{2}\int_\Om |D(u\vee v -u)|^2dx 
\\
&  \Phi(u\wedge v)\ge \Phi(v) + \int_\Om q(u\wedge v-v)dx+\frac{\gamma}{2}\int_\Om |D(u\wedge v -v)|^2dx .
\end{align*}
One now deduces, following the arguments in the proof of Lemma~\ref{lem:control}:
\begin{equation}\label{eq:halfcontrol2}
\begin{aligned}
  &  \gamma\int_{\{v>u\}} |Dv-Du|^2dx \le \int_\Om (q-p)(v-u)^+dx\\
  &  \gamma\int_{\{v<u\}} |Dv-Du|^2dx \le -\int_\Om (q-p)(v-u)^-dx.
\end{aligned}
\end{equation}
Summing, we find:
\[
  \gamma\int_\Om |Dv-Du|^2dx \le \int_\Om (q-p)(v-u)dx.
\]
Using $v=u^{n}$, $u=u^{n+1}$ and~\eqref{eq:ELcompactform}, it follows for all $n\ge 1$:
\begin{multline*}
  \gamma\int_\Om |Du^{n+1}-Du^n|^2dx
\\  \le \int_\Om \left(-\sign(u^n-u^{n-1})\frac{\|u^n-u^{n-1}\|_1}{\tau}+
    \sign(u^{n+1}-u^{n})\frac{\|u^{n+1}-u^{n}\|_1}{\tau}\right)(u^n-u^{n+1})dx
  \\
  \le -\frac{1}{\tau}\|u^{n+1}-u^n\|^2_1 + \frac{1}{\tau}\|u^n-u^{n-1}\|_1\|u^{n+1}-u^n\|_1
\end{multline*}
(there is an abuse of notation here since ``sign'' is multivalued, however
we use only that $|\sign|\le 1$ and $\sign(u^{n+1}-u^n)(u^n-u^{n+1})=-|u^{n+1}-u^n|$),
which we rewrite \new{as}:
\begin{multline}\label{eq:estimspeed}
  \gamma\tau \int_\Om \left|D\frac{u^{n+1}-u^n}{\tau}\right|^2 dx
  + \frac{1}{2\tau^2}(\|u^{n+1}-u^n\|_1-\|u^{n}-u^{n-1}\|_1)^2
  + \frac{\|u^{n+1}-u^n\|_1^2}{2\tau^2}
\\  \le  \frac{\|u^{n}-u^{n-1}\|_1^2}{2\tau^2}.
\end{multline}
Then, summing~\eqref{eq:estimspeed}, we get the estimate:
\begin{equation}\label{eq:h1speed}
  \gamma\int_{0}^{n\tau}\|D\dot{\hat u}_\tau(t+\tau)\|_2^2dt\le
  \frac{\|u^1-u^0\|^2_1}{2\tau^2} \new{\le \frac{1}{2}\|\partial^0\Phi(u^0)\|_\infty^2}.
\end{equation}
  If the initial speed is not bounded we can sum from $m$ to $n>m$ and
  get
  \begin{equation}\label{eq:h1speedunbounded}
  \gamma\int_{m\tau}^{n\tau}\|D\dot{\hat u}_\tau(t+\tau)\|_2^2\,dt\le
  \frac{\|u_\tau((m+1)\tau)-u_\tau(m\tau)\|^2_1}{2\tau^2}\le
  \newer{\frac{\Phi(u^0)}{(m+1)\tau},}
\end{equation}
\newer{thanks to Theorem~\ref{th:controlspeed} \new{(ii.)}.}
Recalling Theorem~\ref{th:controlspeed} we are in particular able to deduce the following result:
\begin{theorem}\label{th:speedH1}
  Assume $F$ is $\gamma$-convex and let $u$ be given by Theorem~\ref{th:minMM}.
  Then $\dot u\in L^2((t,+\infty);H^1(\Om))$ for any $t>0$, with
  \[
    \new{\gamma}\int_t^{+\infty} \new{\|D\dot{u}\|_2^2}\,ds\le\frac{\Phi(u^0)}{t}.
  \]
  If in addition $\partial\Phi(u^0)\neq\emptyset$, then
  \[
    {\gamma}\int_0^{+\infty} \new{\|D\dot{u}\|_2^2}\,ds\le \new{\frac{1}{2}}
    \|\partial^0 \Phi(u^0)\|_\infty^2.
  \]
\end{theorem}

\begin{remark} Taking into account \eqref{eq:halfcontrol2} when deriving~\eqref{eq:halfcontrol}, we can derive slightly
  more precise estimates which may be useful in case the initial speed $q^0\in\partial\Phi(u^0)$
  has a sign. Indeed, we obtain for instance that:
  \begin{itemize}
  \item If $\{ u^1>u^0 \}$ has positive measure, then
  \[
    \frac{\|u^{1}-u^0\|_1}{\tau} \le \ess\sup (-q^0) 
    - \gamma\frac{\int_\Om |D (u^1-u^0)^+|^2 dx}{\|(u^1-u^0)^+\|_1};
  \]
\item If $\{u^1<u^0\}$ has positive measure, then
  \[
    \frac{\|u^{1}-u^0\|_1}{\tau} \le \ess\sup q^0 
    - \gamma\frac{\int_\Om |D (u^1-u^0)^-|^2 dx}{\|(u^1-u^0)^-\|_1}.
  \]
\end{itemize}
In particular, if $q^0\le 0$ a.e., we deduce that $u^1\ge u^0$ a.e., but then
$q^1:= -\sign(u^1-u^0)\|u^1-u^0\|^2/\tau\in \partial \Phi(u^1)$ is also non-positive
and again, $u^2\ge u^1$ a.e.: by induction we find that $u^{n+1}\ge u^n$ for all $n\ge 0$.
\end{remark}

\newer{We now are able to derive rigorously \eqref{eqslope}.
  First, thanks to Poincar\'e inequality, $\dot{u}\in L^2((t,T)\times \Om)$ for any $T>t>0$. Indeed, in the Dirichlet case, one has $\int_\Om |\dot{u}(s)|^2dx\le c_\Om \int_\Om |D\dot{u}(s)|^2dx$ for each $s>0$ such that the right-hand side integral is finite, with $c_\Om$ the Poincar\'e constant of $H^1_0(\Om)$. In the Neumann case, Poincar\'e-Wirtinger's inequality yields $\int_\Om |\dot{u}(s)-m(s)|^2 dx\le c'_\Om\int_\Om |D\dot{u}(s)|^2dx$ for $m(s)=(1/|\Om|)\int_\Om \dot{u}(s)dx$, which is bounded thanks to Corollary~\ref{cor:Lip}.
  
  In particular for any $b>a>0$,
  \[
    \lim_{|s|\to 0} \int_a^b\int_\Om |\dot{u}(t+s,x)-\dot{u}(t,x)|^2dx\, dt=0
  \]
  and using $(u(t+s)-u(t))/s=(1/s)\int_0^s \dot{u}(r)dr$ and Jensen's inequality,
  \[
    \lim_{|s|\to 0} \int_a^b\int_\Om \left|\tfrac{u(t+s,x)-u(t,x)}{s}-\dot{u}(t,x)\right|^2dx\, dt=0.
  \]
  Hence, we can find a sequence $s_k\downarrow 0$ such that for a.e.~$t>0$,
  \[
    \lim_{k\to\infty} \left\|\tfrac{u(t\pm s_k)-u(t)}{\pm s_k}-\dot{u}(t)\right\|_2 = 0
  \]

  Now we consider $q(t)$ from Theorem~\ref{th:dissip}. For a.e.~$t>0$, $q(t)\in L^\infty(\Om)$,
  $\Phi(u(t))$ is differentiable at $t$, and one has for $s$ small (positive or negative):
  \[
    \frac{1}{|s|}\Phi(u(t+s))-\Phi(u(t))\ge -\int_\Om q(t)\frac{u(t+s)-u(t)}{|s|} dx.
  \]
  Choosing $s$ along the sequence $s_k$ and sending $k\to \infty$, we deduce~\eqref{eqslope}.

  As a consequence, one has:
}
\[
\Phi(u(0))-\Phi(u(t)) = \int_0^t \int_\Om q(s)\dot{u}(s)dx ds\le
\frac{1}{2}\int_0^t \|q(s)\|_\infty^2 + \|\dot{u}(s)\|_1^2
\]
which combined with~\eqref{eq:dissip}, yields that $q(s)\in\partial\|\cdot\|_1^2(\dot{u}(s))/2$
 for a.e.~$s>0$.

If $F$ is $\gamma$-convex and $C^1$, \new{with full domain,}
we have additionally that
$q(t) = \Div\nabla F(Du(t))$
for a.e.~$t>0$\new{, \textit{cf}~Lemmas~\ref{lem:subgrad}-\ref{lem:subgradDiri}}.
 Hence we have:
 \begin{theorem} \label{th:contevolsc}
   Assume $F$ is $C^1$ and strongly convex, \new{with full domain,} and let $u$ be
   a limit of minimizing movements given by Theorem~\ref{th:minMM}, starting from $u^0$
   with $\Phi(u^0)<+\infty$. Then,
   $\dot u\in L^2((t,+\infty);H^1(\Om))$ for any $t>0$ and satisfies the equations
\begin{equation}\label{eq:contevolsc}
  \begin{cases} |\Div \nabla F(Du)|\le \|\dot{u}\|_1& \text{a.e.~in }(0,+\infty) \times \Om \\
    \dot{u}\,\Div \nabla F(Du) = |\dot{u}|\|\dot{u}\|_1 & \text{a.e.~in }(0,+\infty) \times \Om.
  \end{cases} 
   \end{equation}
\end{theorem}

\subsection{Minimal surface energy}

The case where $\Phi(u) = \int_\Om\sqrt{1+|D u|^2}dx$ is in between
the \new{setup of the previous} section and \new{that of} the last Section~\ref{sec:per}, where we introduce a geometric
version of this gradient flow.
In that case, we remark that if we
can show that when $u^0$ is $L$-Lipschitz for
some constant $L\ge 0$,
$u$ remains $L$-Lipschitz, then from Section~\ref{sec:sc} we deduce that
the solution satisfies $\dot{u}\in H^1(\Om)$ for positive time and that the
characterization~\eqref{eq:contevolsc} holds. Indeed, in that case, since the gradients
are all bounded by $L$, $F$ is $\gamma$-convex, with $\gamma = (1+L^2)^{-3/2}$.

This is the case for instance if we consider the problem in a periodic setting
($\Om=\R^d/\Z^d$):\begin{lemma}
  Let $\Om=\R^d/\Z^d$, $F(p)=\sqrt{1+|p|^2}$, $v$ a $L$-Lipschitz, \new{($L\ge 0$)} function and $u$
  a minimizer of:
  \begin{equation}\label{eq:graphmin}
    \min_u\Phi(u) + \frac{1}{2\tau}\left(\int_\Om |u-v|dx\right)^2.
  \end{equation}
  Then $u$ is $L$-Lipschitz, and unique.
\end{lemma}
\begin{proof}
  It is enough to show it for the unique solution $u_p$, $p>1$, of:
   \begin{equation}\label{eq:appp}
    \min_u\Phi(u) + \frac{1}{2\tau}\left(\int_\Om |u-v|^pdx\right)^{2/p}
  \end{equation}
  since in the limit $p\to 1$ one recover a minimizer (hence \textit{the} minimizer)
  for $p=1$.

  We first show \new{a comparison result in a simplified setting}: let $v>v'$, let $u$ minimize, for some $\lambda>0$:
  \begin{equation}\label{eq:minp}
    \min_u \Phi(u) + \frac{\lambda}{p}\int_\Om |u-v|^p dx
  \end{equation}
  and let $u'$ solve the same problem with $v$ replaced with $v'$. Then, comparing
  the energy of $u$ with the energy of $u\vee u'$, and the energy of $u'$
  with the energy of $u\wedge u'$ and summing both  \new{inequalities} we end up (using~\eqref{eq:submod})
  with:
  \[
    \int_\Om |u'-v'|^pdx - \int_\Om |u\wedge u'-v'|^pdx \le
    \int_\Om |u\vee u'-v|^pdx -    \int_\Om |u-v|^pdx ,
  \]
  that is:
  \[
    \int_{\{u<u'\}} |u'-v'|^p-|u-v'|^p dx \le \int_{\{u<u'\}} |u'-v|^p-|u-v|^pdx. 
  \]
  One may \new{rewrite this as}:
  \[
    \int_{\{u<u'\}} \int_{u(x)}^{u'(x)} p|t-v'(x)|^{p-2}(t-v'(x)) - p|t-v(x)|^{p-2}(t-v(x))dx\le 0,
  \] 
  which, since $-v(x)<-v'(x)$, is not true unless $u\ge u'$ a.e.

  Now, assume $v$ is $L$-Lipschitz and let $u=u_p$ be the minimizer of~\eqref{eq:appp}.
  For $z\in\R^d$, $\eps>0$,
  let $v'(x) = v(x-z)-L|z|-\eps<v(x)$ and $u'(x)=u(x-z)-L|z|-\eps$ be the solution
  of~\eqref{eq:appp} with $v$ replaced with $v'$. The Euler-Lagrange
  \new{equations} for $v$
  and $v'$ are:
  \[
    \begin{cases}    -\partial \Phi(u) + \frac{1}{\tau}\left(\int_\Om |u-v|^pdx\right)^{2/p-1}|u-v|^{p-2}(u-v)=0,\\
      -\partial \Phi(u') + \frac{1}{\tau}\left(\int_\Om |u'-v'|^pdx\right)^{2/p-1}|u'-v'|^{p-2}(u'-v')=0
    \end{cases}
  \]
  hence letting $\lambda = \|u-v\|^{2-p}/\tau= \|u'-v'\|^{2-p}/\tau$, we find that
  $u$ is a minimizer of~\eqref{eq:minp} while $u'$ is a minimizer of the same problem
  with~$v$ replaced with $v'$. We deduce that $u'\le u$. Sending $\eps\to 0$, it follows that
  \[
    u(x-z)-L|z|\le u(x)\quad \forall x\in\Om, z\in\R^d
  \]
  which shows that $u$ is $L$-Lipschitz. \new{Letting $p\to 1$, we eventually
  find a $L$-Lipschitz solution to~\eqref{eq:graphmin}.}

  We now observe that if there is another \new{minimizer} $u'\in BV(\R^d/\Z^d)$
  \new{of~\eqref{eq:graphmin}}, by \new{strict} convexity arguments,
  the absolutely continuous part of the gradient must be the same as $Du$, and they can
  differ only by a singular part. 
  \new{In addition (Remark~\ref{rem:surfmin}-(\ref{rem-l1u})), $\|u-v\|_1=\|u'-v\|_1$ so that the energy
  of $u'$ is $\int_\Om F(D^a u')\,dx+\int_\Om F^\infty(D^s u') +\|u'-v\|_1^2 /(2\tau)= \int_\Om F(D^a u)\,dx+\int_\Om F^\infty(D^s u') +\|u-v\|_1^2 /(2\tau)$ so that $\int_\Om F^\infty(D^s u')=0$ and $u=u'$ (up to a possible constant, but then using that $v$ is Lipschitz and $\|u-v\|_1=\|u'-v\|_1$ shows that they cannot differ). Hence there is a unique minimizer of~\eqref{eq:graphmin}.}
   
\end{proof}
Hence, we obtain the following result:
\begin{theorem} Let $u^0$ a Lipschitz function over $\Om=\R^d/\Z^d$.
  Then the discrete motion converges to $u(t)\in C^0([0,T];L^p(\Om))$ for any $p<d/(d-1)$,
  with 
  $\int_{s}^\infty |D\dot u|^2dt\le C\Phi(u^0)/s$
  for any $s>0$, and  \new{$u$} satisfies
  \[
    \begin{cases}    |\kappa_u(x)|\le \|\dot u\|_1 & \text{ a.e.~in }\Om,\\
      -\dot{u}\kappa_u(x) = |\dot{u}| \|\dot u\|_1 & \text{ a.e.~in }\Om
    \end{cases}\quad
    \text{ for a.e. } t\ge 0,
  \]
  where $\kappa_u = \Div \left(D u/\sqrt{1+|D u|^2}\right)$ a.e.
\end{theorem}

\begin{remark}
  The proof of the existence of a \new{unique} $L$-Lipschitz solution on the torus when $u^0$ is
  $L$-Lipschitz 
  \new{only relies on the strict convexity of $F$, so that the results in this section
  are also true for any $F$ which is strongly convex on bounded subsets of $\R^d$.}
\end{remark}

\section{Monotone solutions} \label{sec:monotone}
In this section we consider
the case of Dirichlet boundary conditions ($\textup{dom}(\Phi) = {u^0}+ H_1^0(\Om)$),
and we assume that $u^0\in BV(\Om)$ a {\it subsolution}
in the following sense:
\begin{definition}\label{defsub} We say that $u^0\in BV(\Om)$ is a subsolution if
  for any $v\in BV(\Om)$ with $\{v\neq u^0\}\subset\subset \Om$, 
  we have
\[
  v\le u^0\Rightarrow \Phi(v)\ge \Phi(u^0).
\]
\end{definition}

\begin{lemma}\label{profi}
If $u^0$ is a subsolution then,
for any $v\in BV(\Om)$ with $\{v\neq u^0\}\subset\subset \Om$, 
we have
\[
  \Phi(\max\{u^0,v\}) \le \Phi(v).
\]
\end{lemma}
\begin{proof}
Since $u_0$ is a subsolution, we know that $\Phi(\min\{u^0,v\})\ge \Phi(u^0)$.
Recalling that
\[
  \Phi(\min\{u^0,v\})+\Phi(\max\{u^0,v\}) \le \Phi(v)+\Phi(u^0),
\]
it follows that $\Phi(\max\{u^0,v\}) \le \Phi(v)$.

\end{proof}

Replacing $u^0$ with $\max\{u^0,u^1\}$ in the variational problem
which defines $u^1$, we find that $u^1\ge u^0$ a.e. in $\Om$;
in particular, the Euler-Lagrange equation reads:
\[
 \partial\Phi( u^1) + \frac{\|u^1-u^{0}\|_1}{\tau} \new{\varphi} = 0,\quad
  \varphi\in \sign(u^1-u^{0})\quad \text{a.e. in $\Om$.}
\]

\begin{proposition}\label{promono}
If $u^0$ is a subsolution then, for any $n\ge 1$,
$u^{n}\ge u^{n-1}$ and $u^n$ is also a subsolution.
\end{proposition}
\begin{proof}
  This follows the proof
  of a similar result in~\cite{Laux} for mean-convex sets, see also Sec.~\ref{sec:out}.
By Lemma~\ref{profi}, for any $v\in BV(\Om)$ with $\{v\neq u^0\}\subset\subset \Om$, 
we have
\[
  \Phi(\max\{u^0,v\}) \le \Phi(v).
\]
Let $v\in BV(\Om)$ with $\{v\neq u^0\} \subset\subset\Om$, and assume $v\le u^1$.
We have
\begin{align*}
  \Phi(v) \ge \Phi(\max\{u^0,v\}) &
  \ge \Phi(u^1) + \int_\Om -\frac{\|u^1-u^{0}\|_1}{\tau} \varphi (\max\{u^0,v\}-u^1) dx
  \\ &=\Phi(u^1) +\frac{\|u^1-u^{0}\|_1}{\tau}  \int_{\{u^1>u^0\}} (u^1-\max\{u^0,v\}) dx \ge\Phi(u^1)
\end{align*}
showing that $u^1$ is also a subsolution, and the thesis follows by iterating the argument.
\end{proof}

Let us set now $\lambda_1= \|u^1-u^0\|_1/\tau$. We observe that for
any $v\ge u^0$ with $v-u^0$ with compact support, one has
\[
  \Phi(v) \ge \Phi(u^1) 
  + \lambda_1 \int_\Om -\varphi(v-u^1)dx. 
\]
Since $v\ge u^0$, we get 
\begin{multline*}
  \int_\Om -\varphi(v-u^1)dx = \int_{\{u^1>u^0\}} u^1-v \,dx
  + \int_{\{u^1=u^0\}} -\varphi (v-u^0) dx
\\  \ge \int_{\{u^1>u^0\}} u^1-v \,dx   - \int_{\{u^1=u^0\}} v-u^0 \,dx = \int_\Om u^1-v \,dx,
\end{multline*}
and we deduce that
\[
  \Phi(v)+\lambda_1\int_\Om v \,dx \ge \Phi(u^1) + \lambda_1\int_\Om u^1dx.
\]
It follows that $u^1$ is a solution of the obstacle problem (with Dirichlet boundary conditions)
\[
\min_{v\ge u^0}	\ \Phi(v) + \lambda_1\int_\Om v\,dx.
\]
Notice that, if $F$ has superlinear growth and is strictly convex, the solution
is unique.\footnote{When $F$ has linear growth such a statement is unclear, we only know
 that, for all $\lambda_1$ but a countable number, the solution is unique, otherwise
it is trapped in between a minimal and a maximal solution.}
Observe also that, 
if $v,v'$ are minimizers of the above obstacle problem 
for, respectively, two different non-negative parameters $\lambda$ and $\lambda'$, then the inequality
 \[
   \Phi(v)+\lambda \int_\Om v \,dx + \Phi(v') +\lambda'\int_\Om v'\,dx
   \le 
  \Phi(v\wedge v')+\lambda\int_\Om v\wedge v' \,dx+\Phi(v\vee v')
  +  \lambda'\int_\Om v\vee v' \,dx
\]
shows that $(\lambda-\lambda')\int_\Om (v-v')^+dx\le 0$. Hence, if $\lambda>\lambda'$
one has $v\le v'$.


\smallskip

Let us now introduce, for $m\ge 0$,  the volume function
\begin{equation}\label{eq:volf}
  f(m) := \min\left\{ \Phi(v): v\ge u^0, v=u^0\text{ on }\partial\Om,
    \int_\Om v-u^0 dx = m\right\}.
\end{equation}

From now on we shall assume that $F$ is strictly convex and superlinear. 
For any $\lambda\in\R$, we define $v^\lambda$ as the solution of the obstacle problem with parameter
$\lambda$, that is, the unique minimizer of
\begin{equation}\label{eq:obstacle}
  \min_{v\ge u^0, v=u^0\,\partial\Om} \Phi(v) + \lambda \int_\Om v \,dx.
\end{equation}
Observe that $\lim_{\lambda\to+\infty} v^\lambda =u^0$.
Thanks to the uniqueness of the solution and the comparison principle,
we observe that the domain $\mathcal{D}:=\{ (x,\new{z}): x\in \Om, u^0(x)< \new{z} < \sup_{\lambda\le 0}v^\lambda(x)\}$ is such that
for any $(x,\new{z})\in \mathcal{D}$, there is a unique $\lambda>0$ such that $\new{z}=v^\lambda(x)$. Indeed, since both $\sup_{\lambda'>\lambda} v^{\lambda'}$
and $\inf_{\lambda'<\lambda}v^\lambda$ are minimizers of~\eqref{eq:obstacle},
they must coincide for all $\lambda$. In particular, the function $$\lambda\mapsto
\int_\Om v^\lambda-u^0 dx=: m^\lambda$$ is continuous and decreasing,
going from $0$ as $\lambda\to+\infty$, to some maximal
value $\bar m \le +\infty$ as $\lambda\to \new{-\infty}$\footnote{If domain of $F$ is not the entire space, 
  the maximal
reachable mass $\bar m$ can be finite.}.
One can check easily that for any $m$, one has $f(m) = \Phi(v^\lambda)$ for
any $\lambda$ such that $m=m^\lambda$.
On the other hand, if $v'$ is another minimizer of~\eqref{eq:volf},
for $m=m^\lambda$,
then since $\Phi(v')+\lambda\int_\Om v'dx = \Phi(v^\lambda)+\lambda \int_\Om
v^\lambda dx$, $v'$ also minimizes~\eqref{eq:obstacle} and by uniqueness
$v'=v^\lambda$.

In addition, given $m,m'$ and corresponding $\lambda,\lambda'$, we have
\[
  f(m')+\lambda m' = \Phi(v^{\lambda'})+\lambda\int_\Om v^{\lambda'}-u^0dx
  \ge \Phi(v^\lambda)+\lambda \int_\Om v^\lambda-u^0 dx
= f(m)+\lambda m,
\]
showing that 
\begin{equation}\label{eqml}
\textrm{$f$ is convex and $-\lambda\in\partial f(m)$. }
\end{equation}

Observe that, in case $F$ is not superlinear or not strictly convex,
one can still build by approximation an increasing family of minimizers with increasing masses, 
minimizing the obstacle problem for
some non-increasing multipliers, but one might lose uniqueness.


\begin{proposition}\label{prolambda}
Let $F$ be strictly convex and superlinear, and let 
$u^0$ be a subsolution, then $u^n=v^{\lambda_n}$
for any $n\ge 1$, where $\lambda_n:=\|u^n-u^{n-1}\|_1/\tau$.
\end{proposition}
\begin{proof}
By the above analysis $u^n$ is the unique solution of
\[
  \min_{v\ge u^{n-1}} \Phi(v) + \lambda_n\int_\Om v dx
\]
with $\lambda_n=\|u^n-u^{n-1}\|_1/\tau$. We show by induction
that this is also $v^{\lambda_n}$,
knowing that it is true for $n=1$. Assume it holds
for $u^{n-1}$, then by comparison principle and the fact $\lambda_n$ is non-increasing
(see Theorem~\ref{th:controlspeed}), one has $v^{\lambda_n}\ge u^{n-1}$. Hence we get
\[
\Phi(u^n) + \lambda_n\int_\Om u^n dx\le\Phi(v^{\lambda_n}) + \lambda_n\int_\Om v^{\lambda_n} dx.
\]
But since $u^n\ge u^0$, the reverse inequality is also true, hence $u^n$ and $v^{\lambda_n}$
are both minimizers of the obstacle problem for $\lambda_n$. By uniqueness, we deduce that they coincide.
\end{proof}

Observe that \new{the value} $\lambda_n$ can
be also built as follows: given $\lambda_{n-1}$, when $\lambda$ decreases from $\lambda_{n-1}$ to
$0$, then $\|v^\lambda-u^{n-1}\|_1/\tau$ increases from $0$ to $\|v^0-v^{\lambda_{n-1}}\|_1/\tau>0$, and
there is a value in $(0,\lambda_{n-1})$ for which they coincide.
Moreover, letting $m_n = \int_\Om u^n-u^0 dx$, by Proposition~\ref{prolambda} and \eqref{eqml} we have
\begin{equation}\label{gradm}
\frac{m_{n}-m_{n-1}}{\tau} = \lambda_n \in -\partial f(m_n), 
\end{equation}
for any $n\ge 1$,
so that the sequence $(m_n)_n$ solves the discrete implicit Euler scheme for the gradient flow of the convex function $f$.

\begin{theorem}\label{th:monuq}
Let $F$ be strictly convex and superlinear, and let 
$u^0$ be a subsolution with $\Phi(u^0)<+\infty$. Then there exists a unique limit solution $u$ given by Theorem~\ref{th:minMM} with initial datum $u^0$.
Moreover, the function $u$ is non-decreasing in $t$, and $u(t)= v^{\lambda(t)}$ for a.e. $t> 0$, where $\lambda(t)\in L^2((0,+\infty))$ is positive and non-increasing.
\end{theorem}

\begin{proof}
The monotonicity of $u$ in $t$ follows directly from Proposition~\ref{promono}.

Letting $\lambda_\tau(t) = \lambda_{\floor{t/\tau}+1}$ and $m_\tau(t) = m_{\floor{t/\tau}+1}$ for $t\ge 0$,
by Helly's Theorem we may assume that, up to a subsequence,
  $\lambda_\tau$ and $m_\tau$
  converge pointwise  to functions $\lambda(t)$ and $m(t)$  which are respectively non-increasing
  and non-decreasing. 
  By Proposition~\ref{prolambda} we then get that
$u_\tau(t)\to u(t)=v^{\lambda(t)}$   and $\int_\Om u_\tau(t)-u^0dx \to m(t) = m^{\lambda(t)}$ as $\tau\to 0$,
for a.e. $t> 0$.

Recalling \eqref{gradm} we also have that 
$m$ is the unique solution of the gradient flow
\[\dot m + \partial f(m)\ni 0 \]
with initial value $m(0)=0$, see for instance~\cite{BrezisOMM}.
It follows that $u(t)$ is the solution of~\eqref{eq:volf}
for $m=m(t)$, and since the latter is unique, we deduce that also
the limit flow $u(t)$ is unique, and that $u_\tau\to u$ as $\tau\to 0$, without passing to a subsequence.
The fact that
$\lambda\in L^2((0,+\infty))$ \new{follows} by the dissipation
estimate~\new{\eqref{eq:dissip}}.
\end{proof}

\begin{remark}
Observe that
$\|u^{n\tau}-u^{m\tau}\|_1 = \tau\sum_{l=m+1}^n\lambda_l=
\int_{m\tau}^{n\tau}\lambda_\tau(s)ds$, hence 
\[
\|u(t_2)-u(t_1)\|_1=\int_{t_1}^{t_2}\lambda(s)ds\qquad
\textrm{for all $0\le t_1<t_2$,}
\]
which is equivalent to
\begin{equation}\label{eq:vincololambda}
  m(t) = \int_\Om v^{\lambda(t)} (x) - u^0(x)dx = \int_0^t \lambda(s) ds\qquad  \textrm{for all $t\ge 0$}.
\end{equation}
%
If $F$ is of class $C^1$ \new{in $\R^d$ and superlinear}, recalling that the functions $v^\lambda$ satisfy
\[
-\Div\nabla F(Dv^\lambda) +\lambda=0\qquad \textrm{a.e.~in $\{v^\lambda>u^0\}$}
\]
\new{(see Lemma~\ref{lem:subgradDiri})}, equation \eqref{eq:vincololambda} implies \eqref{eq:contevolsc}.
\end{remark}

\begin{remark}
If $F$ is of class $C^1$ \new{in $\R^d$ and superlinear} 
one can check that two different values of $\lambda$ yield different functions
(when $v^\lambda>u^0$, since in that case one has
$-\Div\nabla F(Dv^\lambda) +\lambda=0$ a.e.~in $\{v^\lambda>u^0\}$). Then, using that
$t\mapsto u(t)=v^{\lambda(t)}$ is H\"older continuous in $L^1(\Om)$ by Theorem~\ref{th:minMM},
it follows that $\lambda(t)$ is continuous. 
\end{remark}

\section{The Dirichlet energy}\label{sec:dirichlet}
In this section, we consider the simplest case $F(\xi)=|\xi|^2/2$, so that
\[
  \Phi(u) = \frac{1}{2}\int_\Om |D u|^2
\]
is the Dirichlet energy of $u$.

In what follows, we shall consider either the case of Dirichlet boundary conditions ($\textup{dom}(\Phi) = {u^0}+ H^1_0(\Om)$),
or the case of homogeneous Neumann boundary conditions ($\textup{dom}(\Phi) = H^1(\Om)$).


\subsection{Uniqueness} 
Assuming that $\Phi(u^0) <+\infty$, the limit solution $u$ provided by Theorem~\ref{th:minMM} satisfies
\begin{equation}\label{eq:dissipDiri}
  \int_\Om |D u(t,x)|^2dx + \frac{1}{2} \int_0^t \|\dot{u}(s)\|_1^2 ds
  + \frac{1}{2}\int_0^t \|\Delta u(s)\|^2_\infty ds \le \int_\Om|D u^0|^2dx,
\end{equation}
which is~\eqref{eq:dissip}, and we take into account (\textit{cf}~Theorem~\ref{th:speedH1}) that $\dot{u}\in L^\infty([t,+\infty];H^1(\Om))$ for any $t>0$,
and $\partial\Phi(u(t))=\{-\Delta u(t)\}$ for a.e.~\new{$t\in (0,+\infty)$}.
As usual, this can be rewritten:
\[
  \int_0^t \left(\int_\Om D u(s,x)\cdot D \dot{u}(s,x) dx
  + \frac{1}{2} \|\dot{u}(s)\|_1^2 
  + \frac{1}{2} \|\Delta u(s)\|^2_\infty \right)ds\le 0,
\]
which \new{yields} $\Delta{u}(s)\in \|\dot{u}\|_1\sign(\dot{u})$
a.e. in $\Om$, for a.e.~$s\in [0,T]$ \new{(for any $T>0$)}, and we obtain the equations
\begin{equation}\label{eq:contevol}
  \begin{cases} |\Delta u|\le \|\dot{u}\|_1& \text{a.e.~in }\Om \\
    \dot{u}\Delta u = |\dot{u}|\|\dot{u}\|_1 & \text{a.e.~in }\Om
  \end{cases} \quad\textup{ a.e.~in } [0,T],
\end{equation}
\textit{cf}~\eqref{eq:contevolsc}.

It turns out that this defines a \textit{unique} evolution starting from
$u^0\in H^1(\Om)$. 
Indeed, for different solutions $u,v$ of \eqref{eq:contevol} we have
\begin{equation}\label{eq:shrinkDu}
  \frac{d}{dt}\int_\Om |D u-D v|^2dt
  = 2\int_\Om (D u-D v)\cdot D (\dot{u}-\dot{v})dx
  =-2\int_\Om (\Delta u-\Delta v)\cdot (\dot{u}-\dot{v})dx\le 0.  
\end{equation}

\begin{theorem}\label{eq:unidiri}
  For any $u^0\in H^1(\Om)$, there is a unique flow $u\in C^0([0,+\infty);H^1(\Om))$
  which solves~\eqref{eq:contevol}. In addition, the minimizing movements $\hat u_\tau$ converge
  to $u$ in $C^0((0,+\infty);H^1(\Om))$ (\new{\emph{i.e.}}, locally uniformly in
  time), as $\tau\to 0$. The semi-norm of the speed
  $\|D\dot u\|_2$ is non-increasing in time.
\end{theorem}

\begin{proof} The uniqueness follows from~\eqref{eq:shrinkDu} in the case
  of Dirichlet boundary conditions. In the case of Neumann conditions,
  assume we have two different solutions $u(t)$ and $u(t)+c(t)$, $c(t)\in \R$ (with $c\in H^1(\R_+)$
  thanks to~\eqref{eq:dissipDiri}).
  The equations state, then, that for a.e.~$t$ and a.e.~$x\in\Om$,
  both  $\dot{u}\Delta u = |\dot{u}|\|\dot u\|_1$ and
  $(\dot{u}+\dot{c})\Delta u = |\dot{u}+\dot{c}|\|\dot u+\dot{c}\|_1$.
  In case $\{\Delta u=0\}$ has positive measure, we deduce that either $\|\dot u\|_1=0$
  but then $0=|\dot c|^2|\Om|$ on a set of positive measure, meaning $\dot c=0$,
  or $|\dot u|=0$ on a set of positive measure and on the same set, $|\dot c| \|\dot u+\dot{c}\|_1=0$.
  Hence again, $\dot c =0$ (or we would have $\dot u\equiv -\dot{c}\neq 0$ a.e., a contradiction).

  Hence, we assume $|\{\Delta u=0\}|=0$. Since in addition, $\int_\Om \Delta u=0$ (in the
  case of Neumann conditions), $\Om$ is split into two sets $\Om^\pm$ of positive measure,
  with $\Delta u>0$ a.e.~in $\Om^+$, hence $\dot u\ge 0$, and $\Delta u<0$ a.e.~in $\Om^-$,
  hence $\dot u+\dot c\le 0$. This contradicts $\dot u\in H^1(\Om)$ if $\dot c>0$, since one would
  have $|\{-\dot c\le \dot u \le 0\}|=0$. Symmetrically, one cannot have $\dot c<0$, hence
  $\dot c=0$. We deduce that $c(t)=0$ (since $c$ must be continuous with $c(0)=0$), and
  \new{this} proves uniqueness in the Neumann case.

  The convergence of $\hat u_\tau$ to the unique possible limit $u$ is guaranteed by
  Theorem~\ref{th:minMM} and Proposition~\ref{prop:interpol}, at least in
  $C^0([0,T];L^p(\Om))$ for any $T>0$ and $p<d/(d-1)$. In addition, it
  follows from \eqref{eq:h1speedunbounded} by standard arguments
  that $\|D\hat u_\tau(t)-D\hat u_\tau(s)\|^2_2 \le 2/ \min\{t,s\}\int_\Om |Du^0|^2dx$,
  from which we also deduce the uniform convergence on any interval $[t,T]$, $T>t>0$.  
  
  Eventually, the fact that the speed is non-increasing in $H^1$ follows from the fact that,
  using~\eqref{eq:shrinkDu},
  $\|Du(t+\eps)-Du(t)\|_2\le \|Du(s+\eps)-Du(s)\|_2$ for any $t>s>0$ and any $\eps>0$.
\end{proof}

We observe that the contraction property in the \new{time-}continuous setting also has a counterpart
for the discrete flow:
\begin{lemma}\label{lem:contract}
  Let $v,v'\in L^1(\Om)$, $v-v' \in H^1(\Om)$ (resp. $H^1_0(\Om)$ in the case of Dirichlet boundary conditions)
  and assume $u$ is a minimizer of
  \[
    \frac{1}{2\tau}\|u-v\|_1^2 + \Phi(u)
  \]
  and $u'$ a minimizer of the same problem with $v$ replaced with $v'$.
  Then:
\begin{equation}\label{eq:H1contract}
  \frac{1}{2} \int_\Om  |D (u- u')|^2dx\le
  \frac{1}{2} \int_\Om |D (v-v')|^2 dx
  - \frac{1}{2} \int_\Om |D (u-u'-v+v')|^2 dx.
\end{equation}
  and in particular
  \[
    \|D u -D u'\|_2 \le \|D (v-v')\|_2.
  \]
\end{lemma}
\begin{proof}
  Subtracting the Euler-Lagrange equations for $u$ and $u'$, multiplying
  by $(u-v)-(u'-v)$ and integrating by parts, we get:
\[
  \int_\Om  (D u-D u')\cdot(D (u - v  - u' + v') dx\le 0
\]
thanks to the monotonicity of the subgradient of $\|\cdot\|^2_1/2$.
It follows
\[
  \int_\Om  |D (u- u')|^2dx\le \int_\Om D (u-u') \cdot D (v-  v') dx
\]
from which we deduce~\eqref{eq:H1contract}.
\end{proof}

  Specializing~\eqref{eq:H1contract} to the case $v=u^n$, $v'=u^{n-1}$,
  for $n\ge 1$, we find:
  \begin{equation}\label{eq:H1decrease}
   \int_\Om  |D (u^{n+1}- u^n)|^2dx\le
   \int_\Om |D (u^{n}-u^{n-1})|^2 dx
  - \int_\Om |D (u^{n+1}-2u^n+u^{n-1})|^2 dx.
  \end{equation}
  which shows that also for the discrete flow one has that $\|D\dot{\hat u}_\tau\|_2$ is non-increasing
  in time.

  \subsection{Energy decay estimate}
  In the case of homogeneous Neumann or Dirichlet boundary conditions, we expect that
  $\lim_{t\to\infty}\int_\Om |Du(t)|^2dx=0$. Actually, we
  \new{we can even provide the rate of convergence}:
  \begin{proposition}
    Let $u$ solve~\eqref{eq:contevol}, with $u^0\in H^1(\Om)$ (with Neumann boundary conditions) or $u^0\in H^1_0(\Om)$ (with Dirichlet boundary conditions). 
    Then, for any $t>0$ we have
\[
  \int_\Om |D u(t)|^2 dx \le
  e^{-\frac{t}{c_\Om}}\int_\Om |D u^0|^2 dx,
\]
where $c_\Om$ is the constant in the Poincar\'e-Wirtinger (Neumann) or Poincar\'e (Dirichlet)
inequality.    
\end{proposition}
\begin{proof}
  We consider the minimizing movement scheme, and, given $\tau>0$, $n\ge 1$,
  we compare the energy of $u^n$ with the energy of
  $u_a:=u^{n-1}+a\tau (u^{n-1}-m^{n-1})$, $a\in\R$,
  where $m^{n-1} = \int_\Om u^{n-1} dx/|\Om|$ is the average of $u^{n-1}$
for Neumann boundary conditions, and $0$ for homogeneous Dirichlet boundary conditions.
One has in particular, thanks to the Poincar\'e(-Wirtinger) inequality:
\[
  \|u^{n-1}-u_a\|_1^2 = a^2\tau^2\|u^{n-1}- m^{n-1}\|_1^2
  \le a^2\tau^2 c_\Om \int_\Om |D u^{n-1}|^2.
\]
Hence:
\begin{multline*}
  \frac{1}{2}\int_\Om |D u^n|^2 dx + \frac{1}{2\tau}\|u^n-u^{n-1}\|_1^2
\\  \le (1+2a\tau +\tau^2a^2)\frac{1}{2}\int_\Om |D u^{n-1}|^2 dx +
  \frac{\tau}{2} a^2 c_\Om \int_\Om |D u^{n-1}|^2.
\end{multline*}
Choosing $a=-1/c_\Om$, we find:
\[
  \frac{1}{2}\int_\Om |D u^n|^2 dx
  \le \frac{1}{2}\int_\Om |D u^{n-1}|^2\left(1-\tfrac{\tau}{c_\Om}+
  \tfrac{\tau^2}{c_\Om^2}\right).
\]
Hence,
\[
  \frac{1}{2}\int_\Om |D u^n|^2 dx
  \le \frac{1}{2}\int_\Om |D u^{0}|^2\left(1-\tfrac{\tau}{c_\Om}+
  \tfrac{\tau^2}{c_\Om^2}\right)^n.
\]
We conclude using that
\[
  \lim_{n\to\infty, n\tau\to t}\left(1-\tfrac{\tau}{c_\Om}+
  \tfrac{\tau^2}{c_\Om^2}\right)^n = e^{-\frac{t}{c_\Om}}.
\]
\end{proof}

\section{Gradient flow of anisotropic perimeters}\label{sec:per}

Given a norm $\p$ on $\R^d$, we consider the anisotropic perimeter
\[
E\mapsto P_\p(E):=\int_{\partial E}\p(\nu)\,d\H^{d-1}.
\]
Letting $\p^o$ be the dual norm of $\p$, we recall that the convex set
\[
W_\p := \{x\in\R^d:\ \p^o(x)\le 1\},
\]
usually called {\it Wulff Shape}, is the unique volume-constrained minimizer of $P_\p$, up to translations and dilations.
We say that $\p$ is smooth (resp. elliptic) if the function $\p^2/2$ is smooth (resp. strongly convex).

We now introduce the geometric $L^1$-minimizing movement scheme.
Given $\tau>0$ and $E\subset\R^d$, we consider the minimum problem
\begin{equation}\label{pbmin}
 \min_{F} P_\p(F) + \frac{1}{2\tau} |E\triangle F|^2,
\end{equation}
and we let $T_\tau E$ be a (\new{possibly non}-unique) minimizer of \eqref{pbmin}.

Given an initial set $E^0\subset\R^d$,  for all $n\in\mathbb N$ we let $E^n:= T_\tau^n E_0$.
For $(t,x)\in  (0,+\infty)\times \R^d$, we also let
\[
E_\tau(t):= E^{\lfloor t/\tau\rfloor} \qquad u_\tau(t,x):=\chi_{E_\tau(t)}(x).
\]
The function $t\mapsto E_\tau(t)$ is the discrete $L^1$-gradient flow of $P_\p$, with initial datum $E^0$.

We point out that the analogous concept for the $L^2$-gradient flow of $P_\p$,
where \eqref{pbmin} is replaced by the problem
\begin{equation}\label{pbminatw}
 \min_{F} P_\p(F) + \frac{1}{2\tau} \int_{E\triangle F} {\rm dist}(x,\partial E)\,dx,
\end{equation}
was originally introduced in \cite{ATW,LS} as a discrete approximation of the mean curvature flow. \new{It is shown in these references that the ``distance term'' in~\eqref{pbminatw} is indeed a (non symmetric) approximation of the squared $2$-distance between the boundaries of $E$ and $F$ (a ``flat'' version could be something like $\int_x \int_{t\in [e(x),f(x)]} |t-e(x)| dt\,dx = \frac{1}{2}\int_x |e(x)-f(x)|^2dx$), at least when these boundaries are smooth enough.}

From \eqref{pbmin} it follows that
\begin{equation}\label{estperimeter}
P_\p(E^n)\le P_\p(E^{n-1})\qquad \text{and}\qquad 
\frac {1}{2\tau} |E^n\triangle E^{n-1}|^2 \le P_\p(E^{n-1})-P_\p(E^n),
\end{equation}
for all $n\in\N$, so that 
\begin{eqnarray}\nonumber
|E^n\triangle E^m|^2
&\le&  \left(\sum_{k=m+1}^n |E^k\triangle E^{k-1}|\right)^2 
\\ \label{estimate}
&\le& 2\tau(n-m) \sum_{k=m+1}^n \left(P_\p(E^{k-1})-P_\p(E^k)\right)
\\ \nonumber
&\le& 2\tau(n-m) P_\p(E^0),
\end{eqnarray}
for all $0\le m<n$.

Reasoning as in the proof of Theorem~\ref{th:minMM}, from \eqref{estperimeter} and \eqref{estimate} we get the following result.

\begin{theorem}\label{existgeom}
Assume that $E^0\subset\R^d$ is a set of finite perimeter. Then there exist a sequence $\tau_k\to \infty$ and a function 
$u(x,t)\in L^\infty((0,+\infty),BV(\R^d))\cap C^{1/2}((0,+\infty),L^1(\R^d))$,
with $u(x,t) = \chi_{E(t)}(x)$ for some family of sets $E(t)$,  such that 
\[
\lim_{k\to +\infty} \sup_{t\in [0,T]}\|u(\cdot,t)-u_{\tau_k}(\cdot,t)\|_{L^1(\R^d)} = 0
\qquad \forall T>0.
\]
\end{theorem}

Following~\cite[Lemma~1.3, Remark~1.4]{LS},
we show a density estimate for minimizers of~\eqref{pbmin}. 

\begin{lemma}\label{density}
  There exists $c>0$ depending only on $\p$ and the dimension $d$ such that the following holds:
  let $E\subset\R^d$ and $F$ a minimizer of~\eqref{pbmin}, then
  \begin{enumerate}
  \item for a.e.~$x\in F\setminus E$
    and all $r>0$ such that $|B(x,r)\cap E|=0$, \new{we have} $|B(x,r)\cap F|\ge cr^d$;
  \item for a.e.~$x\not \in E$, and $r>0$ such that $|B(x,r)\cap E|=0$, if
    $|B(x,r)\cap F|\le cr^d/2$, then $B(x,r/2)\cap F =\emptyset$.  \label{density2}
  \end{enumerate}
\end{lemma}

\begin{proof}
  Following~\cite{LS} we compare the energy of $F$ and $F\setminus B(x,r)$ in~\eqref{pbmin}
  and, introducing $b>a>0$ s.t.~$a|x|\le \p(x)\le b|x|$, we observe that for a.e.~$r>0$,
  if $|B(x,r)\cap E|=0$:
  \[
    P_\p(F) + \frac{1}{2\tau} |E\triangle F|^2 \le P_\p(F\setminus B(x,r))
   + \frac{1}{2\tau} |(E\triangle F) \setminus B(x,r)|^2 
 \]
 implies $a\H^{d-1}(\partial F\cap B(x,r)) \le b \H^{d-1}(\partial B(x,r)\cap F)$.
 Introducing $f(r)=|F\cap B(x,r)|$, this is rewritten \new{as} $(a+b)\H^{d-1}(\partial (F\cap B(x,r)))
 \le f'(r)$, and using the isoperimetric inequality we deduce that for some constant
 $\gamma>0$ depending only on $d,a,b$, \new{there holds} $\gamma f(r)^{1-1/d}\le f'(r)$. The thesis
 follows from a version of Gronwall's Lemma.  
\end{proof}

\begin{remark} A symmetric statement holds for points $x\in E$, with $B(x,r)\cap F$ replaced
with $B(x,r)\setminus F$.
\end{remark}

\begin{remark} A similar proof (see~\cite{LS} again)
  shows that there exists $r(\tau)>0$ such that
  for $r<r(\tau)$, for a.e.~$x\in F$, $|B(x,r)\cap F|\ge cr^d$, and
  for a.e.~$x\not\in F$, $|B(x,r)\cap F^c| \ge cr^d$. In particular,
  the points of Lebesgue density $1$ (resp. $0$) of $F$ form an open set,
  the reduced boundary of $F$ is $\H^{d-1}$-essentially closed, and there is no
  abuse of notation in denoting it $\partial F$.
\end{remark}

\subsection{Outward minimizing case} \label{sec:out}

\begin{definition}
Let $\Omega\subset\R^d$ be open. We say that a set $E\subset\Omega$ is  outward minimizing if 
\[
P_\p(E)\le P_\p(F) \qquad\forall F\supset E, F\subset \Omega.
\]
\end{definition}

Notice that if $\p$ is smooth and $E$ is an outer minimizer with boundary of class $C^2$ then $E$ is $\p$-mean convex, that is, $H_\p(x)\ge 0$ for any $x\in \partial E$,
where $H_\p(x)$ is the $\p$-mean curvature of $\partial E$ at $x$ (see for instance \cite{ChambolleNovaga} for a precise definition).
Conversely, if $H_\p(x)\ge\delta>0$
for any $x\in\partial E$ one can build $\Omega\supset\supset E$ such that $E$ is
outward minimizing in $\Omega$. Notice also that a convex set is always outward minimizing.

We recall the following result proved in \cite[Lemma 2.5]{Laux} (see also \cite[Section 2.1]{ChambolleNovaga}).

\begin{lemma}\label{outlem}
E is outward minimizing if and only if
\[
P_\p(E\cap F)\le P_\p(F)\qquad\forall F\subset \Omega.
\]
\end{lemma}

From Lemmas~\ref{density} and~\ref{outlem} we obtain the following result.

\begin{proposition}\label{lemoutward}
  Assume that $E\subset\subset\Omega$ is outward minimizing in $\Omega$. Then, for $\tau$
  small enough (depending only on $\p$,  $\dist(E,\partial\Omega)$, and the dimension)
  we have that $T_\tau E\subseteq E$ and $T_\tau E$ is outward minimizing in $\Omega$.
  
 In particular, the limit flow obtained in Theorem~\ref{existgeom} is non-increasing and outward minimizing in $\Omega$.

\end{proposition}

\begin{proof}
The first assertion follows from the minimality of $T_\tau E$ and from the fact that
\[
P_\p(T_\tau E\cap E) + \frac{1}{2\tau} |(T_\tau E\cap E)\triangle E|^2
\le P_\p(T_\tau E) +  \frac{1}{2\tau} |T_\tau E\triangle E|^2,
\]
with equality iff $T_\tau E\subseteq E$. We use here Lemma~\ref{outlem} which holds if we can
prove first that $T_\tau E\subset \Omega$. Let $r:=\dist(E,\partial\Omega)/2$.
Then for $\tau$ small enough, we have (comparing the energy of $T_\tau E$
and $E$ in~\eqref{pbmin}) that $|T_\tau E\setminus E|\le \sqrt{2\tau P_\p(E)}\le cr^d/2$
where $c$ is the constant in Lemma~\ref{density}. Using point~\eqref{density2} in
Lemma~\ref{density}, it follows that $\{x: r/2<\dist(x,E)<3r/2\}\cap T_\tau E=\emptyset$ and
we deduce that $T_\tau E\subset \{x:\dist(x,E)\le r/2\}\subset\subset\Omega$.

In order to prove the second assertion, 
we fix $F$ such that $T_\tau E\subset F\subset\Omega$, and we notice that
\[
P_\p(T_\tau E)\le P_\p(F\cap E)  + \frac{1}{2\tau} |(F\cap E)\triangle E|^2 - \frac{1}{2\tau} |T_\tau E\triangle E|^2 \le P_\p(F\cap E) \le P_\p(F),
\]
where the last inequality follows from the outward minimality of $E$.

\end{proof}

\begin{remark} 
From Proposition~\ref{lemoutward} and \eqref{pbmin} it follows that the set $T_\tau E$ solves the minimum problem
\begin{equation}\label{eqoutward}
\min_{F\subset E} P_\p(F) - \frac 1\tau |E|\, |F| + \frac{1}{2\tau} |F|^2,
\end{equation}
hence $T_\tau E$ is also a solution of the volume-constrained isoperimetric problem
(see also~\eqref{eq:volfgeom} later on)
\begin{equation}\label{eqconstr}
\min_{F\subset E, |F|=|T_\tau E|} P_\p(F).
\end{equation}

If $\p$ is smooth and elliptic (that is, $\p^2/2$ is smooth and strongly convex),
from \eqref{eqoutward} it follows that $T_\tau E\cap \textup{int}(E)$  is smooth and
satisfies the Euler-Lagrange equation
\begin{equation}\label{eqEL}
H_\p(x) = \frac{|E\setminus T_\tau E|}{\tau} \qquad\text{for }x\in \partial T_\tau E\cap \textup{int}(E).
\end{equation}
If in addition $\partial E$ is of class $C^{1,1}$, by classical regularity results for the obstacle problem \cite{Caffarelli,CCN}
$\partial T_\tau E$ is also of class $C^{1,1}$ outside a closed singular set of Hausdorff dimension $d-2$,
and satisfies the Euler-Lagrange inequality
\begin{equation}\label{eqELdue}
0\le H_\p(x) \le \frac{|E\setminus T_\tau E|}{\tau} \qquad\text{for a.e. }x\in \partial T_\tau E.
\end{equation}

Passing to the limit in \eqref{eqEL} and  \eqref{eqELdue} as $\tau\to 0$, and reasoning as in Theorem~\ref{th:contevolsc},
we may expect that the limit flow $E(t)$ satisfies the equations
\begin{equation}\label{eq:geometricequation}
  \begin{cases} 0\le H_\p\le -\dfrac{d}{dt} |E(t)|& \text{a.e.~on }\partial E(t) \\
  \\
   H_\p = -\dfrac{d}{dt} |E(t)| & \text{a.e.~on }\partial E(t)\cap \textup{int}(E^0),
  \end{cases} 
   \end{equation}
for a.e. $t>0$.
\end{remark}

\begin{remark} We cannot expect that there always exists a constant
$\lambda>0$ such that $T_\tau E$ is a solution of 
\begin{equation}\label{minlambda}
\min_{F\subset E} P_\p(F)-\lambda|F|,
\end{equation}
as it happens in the case of functions (see Section~\ref{sec:monotone}).
In the sequel we shall see that, in the case $E$ is
convex, this is true only if $E$ does not coincide with its Cheeger set,
and $\tau$ is small enough so that $\lambda$ is greater than 
the Cheeger constant of $E$.
\end{remark}

\subsection{Convex case} 
We now consider the special case of a convex initial set.

\begin{proposition}\label{proconvex}
Let $d=2$ and assume that $\p$ is smooth and elliptic.
Assume also that $E^0$ is a bounded convex set. 
Then the limit flow $E(t)$ obtained in Theorem~\ref{existgeom} is given by a decreasing family of convex subsets of $E^0$.
\end{proposition}

\begin{proof}
As in \cite[Section 4.1]{RossiStefanelliThomas} one can easily show that
$T_\tau E$ is a convex subset of $E$ with boundary of class $C^{1,1}$, satisfying \eqref{eqEL} and \eqref{eqELdue}.
As in \cite[Section 9]{BellettiniNovagaPaolini} (see also \cite[Theorem 2.3]{LSa}), it follows that
each connected component of $\partial T_\tau E\cap \textup{int}(E)$ is a graph and it is contained in $r\partial W_\p$, with $r=\tau/|E\setminus T_\tau E|$.
As a consequence, if $r$ is greater than the inradius of $E$, then 
\begin{equation*}
T_\tau E = E^-_{r} := \bigcup_{x+ rW_\p:\, (x+rW_\p)\subset E}  (x+rW_\p),
\end{equation*}
otherwise $T_\tau E = rW_\p + s$ for some segment $s\subset E$,
and $s$ is a point if $r$ is smaller than the inradius of $E$.

By iterating the previous argument, and taking the limit as $\tau\to 0$, get the thesis.

\end{proof}

\begin{remark} By the argument above we get that 
\[
E(t) = E^-_{r(t)} \qquad t\in [0,T],
\]
where $r(t)$ is continuous, increasing, and $T\ge 0$ is the first time such that $r(T)$ equals the inradius of $E$. 
In particular, the limit flow is unique on $[0,T]$.
\end{remark}

\begin{remark}
By approximating a general norm $\p$ with a sequence of smooth and elliptic norms, 
following the proof of Proposition~\ref{proconvex},
we obtain that there exists $r>0$ such that  
$T_\tau E = E^-_{r}$ or $T_\tau E = rW_\p + s$ for some segment $s\subset E$.

As a consequence, also in the case of a general norm, there exists at least one limit flow $E(t)$ 
given by a decreasing family of convex subsets of $E^0$. We point out that, in the general case, we do not prove uniqueness of the limit flow.
\end{remark}

In \cite{CCMN} it has been proved that, in any dimension $d\ge 2$, a volume-constrained minimizer of $P_\p$ inside a convex set $E$
is unique and convex if its volume is greater or equal than the volume of the Cheeger set of $E$\new{.} \new{We recall that the
  Cheeger set of $E$ is the minimizer $F^{\star}$ of the 
variational problem
\[
\min_{F\subset E}\frac{P_\p(F)}{|F|} =: \lambda^{\star} 
\]
($\lambda^{\star}$ is called the Cheeger constant of $E$). This Cheeger set is unique when $E$ is convex \cite{KLR,CCN,AC}.}
The Cheeger set is also characterized as the largest minimizer ($\emptyset$ being the smallest one)
of the problem $\min_{F\subset E} P_\p(F)-\lambda^{\star}|F|$, which has value $0$.

For $\lambda>\lambda^{\star}$, there is a unique minimizer $F^\lambda$ to~\eqref{minlambda},
which is convex, and coincides with the above volume-constrained minimizer
(and is continuous with respect to $\lambda$, see~\cite{CCMN}). 
Moreover, if $\p$ is smooth and elliptic, $\lambda$ coincides with the mean curvature $H_\p$ of $\partial F^\lambda\cap \textup{int}E$ (otherwise
it can be thought of as a variational mean curvature).

It follows that,
as long as $|E^n|\ge |F^{\star}|$, where $F^{\star}$ is the Cheeger set of $E^0$,
we can define a non-increasing sequence $\lambda_n\ge \lambda^{\star}$
such that $E^n = F^{\lambda_n}$ and which satisfies, for $n\ge 1$,
\[
  \frac{|E^{n-1}|-|E^n|}{\tau} = \lambda_n,
\]
or equivalently for all $n\ge 1$,
\[
  |E^0|-|E^n| = \tau\sum_{k=0}^n \lambda_k.
\]
In the limit $\tau\to 0$, similarly to Section~\ref{sec:monotone},
up to a subsequence the non-increasing function $\lambda_{\floor{t/\tau}+1}$ converges
pointwise to a non-increasing function $\lambda(t)$, while $E_\tau(t)$ converges to $F^{\lambda(t)}$.
In particular, in the limit we find that:
\begin{equation}\label{eqlambint}
  |E^0|-|E(t)| = \int_0^t \lambda(s)ds
\end{equation}
for all $0\le t\le T^{\star}$, where $\lambda(T^{\star})=\lambda^{\star}$. 


If $\p$ is smooth and elliptic, since the sets $F^\lambda$ are all different, the function $t\mapsto \lambda(t)$ is continuous on $[0,T^{\star}]$, 
so that the function $t\mapsto |E(t)|$ is of class $C^1$ by~\eqref{eqlambint}. We deduce that~\eqref{eq:geometricequation} holds for all $t\in (0,T^{\star})$.

We then obtain a partial extension of Proposition~\ref{proconvex} to arbitrary dimensions and for a general norm $\p$.

\begin{proposition}\label{proconvexdue}
Assume that $E^0$ is a bounded convex set not coinciding with its Cheeger set. Then there exists $T^{\star}>0$
such that limit flow $E(t)$ is given by a decreasing family of convex subsets of $E^0$ for $t\in [0,T^{\star}]$.
Moreover, each set $E(t)$ is a volume-constrained minimizer of $P_\p$ inside $E^0$, and $E(T^{\star})=F^{\star}$ is the Cheeger set of $E^0$.
In particular, for $t\in (0,T^{\star}]$ $E(t)$ is the unique minimizer $F^{\lambda(t)}$ of~\eqref{minlambda} for
some $\lambda(t)>\lambda^{\star}$ which solves \eqref{eqlambint}. 
\end{proposition}

For  $m\in [0,|E^0|]$ we let 
\begin{equation}\label{eq:volfgeom}
f(m) := \min\left\{ P_\p(F): \, F\subset E^0,  |E^0\setminus F| = m\right\}.
\end{equation}
Reasoning as in Section~\ref{sec:monotone} we have that, for any $m\in [0,|E^0\setminus F^{\star}|]$ there exists a unique 
$\lambda^m\ge\lambda^{\star}$ 
such that $|E^0\setminus F^{\lambda^m}|=m$ and $P_\p(F^{\lambda^m})=f(m)$.
Moreover the function $m\to \lambda^m$ is non-increasing in this interval, and
\begin{equation}\label{eq:lalla}
-\lambda^m\in \partial f(m),
\end{equation}
which implies that $f$ is convex on $[0,|E^0\setminus F^{\star}|]$. 
With almost the same proof as Theorem~\ref{th:monuq},
we can show the following uniqueness result for the limit flow $E(t)$.
\begin{proposition}\label{proconvexunique}
Assume that $E^0$ is a bounded convex set not coinciding with its Cheeger set. Then the flow
$E(t)$ given by Proposition~\ref{proconvexdue} is unique and satisfies
\begin{equation}\label{eq:lallabis}
 \frac{d |E(t)|}{dt} = -\lambda(t) 
\end{equation}
for all  $t\in (0,T^{\star})$, where $\lambda(t)$ coincides with the mean curvature of $\partial E(t)$ inside $E^0$ and $E(T^{\star})$ is the Cheeger set of $E^0$.
\end{proposition}

\begin{remark} 
If $\varphi$ is smooth and elliptic, from \eqref{eq:lallabis} it follows that $E(t)$ satisfies \eqref{eq:geometricequation}.
\end{remark}

Recalling the proof of Proposition~\ref{proconvex}, when $d=2$ the minimizer $E^m$ in \eqref{eq:volfgeom} is uniquely characterized 
and coincides with the set $E^-_{r^m}$ as long as 
$m\ge |E^-_{r^0}|$, where $r^0$ is the inradius of $E^0$ and $r^m\ge r^0$ is such that $|E^-_{r^m}|=m$. 
When $m<|E^-_{r^0}|$ the minimizer $E^m$ is only unique up to translations.

If in addition $\p(x)=|x|$, it has been proved in \cite{LSa} that the function $f$ 
is convex on $[0,m^0]$, where $m^0=|E^0|-|B_{r^0}|$ 
and $E^m$ is a solution of \eqref{minlambda} with $\lambda=1/r_m$, among sets of volume greater of equal to $|B_{r_m}|$.
Observing also that 
$f(m) = 2\sqrt{\pi (|E^0|-m)}$ for $m\in [m_0,|E^0|)$, reasoning as above 
we get that $f$ satisfies \eqref{eq:lalla} for all $m\in (0,|E^0|)$,
so that we can partly extend the result in Proposition~\ref{proconvexunique}.

\begin{proposition}\label{proconvexuniquebis}
Let $d=2$,  $\p(x)=|x|$, and assume that $E^0$ is a bounded convex set. 
Then the flow $E(t)$ is defined on a maximal time interval $[0,T_{\rm max})$,
with 
\[
\lim_{t\to T_{\rm max}}|E(t)|=0,
\]
it is unique up to translations,
and satisfies \eqref{eq:lallabis} for all $t\in (0,T_{\rm max})$,
where $\lambda(t)$ coincides with the curvature of $\partial E(t)$ inside $E^0$.
Moreover, $E(t)$ is unique as long as $|E(t)|\ge |E^-_{r^0}|$.
\end{proposition}

We show with two simple examples that  uniqueness of the flow cannot be expected for $t>T^{\star}$.
In the following we fix $d=2$ and $\p(x)=|x|$. 

\smallskip

\noindent{\it Example 1.} Let $E^0=B_R(x_0)$ for some $R>0$ and $x_0\in\R^2$. 
Then, by the isoperimetric inequality, $T_\tau E^0$ is a ball contained in $E^0$ of radius $r$ minimizing the function
\begin{equation}\label{funball}
r\mapsto 2\pi r + \frac{\pi^2}{2\tau} (R^2-r^2)^2,
\end{equation}
that is, $r = R- \tau/(2\pi R^2) + o(\tau)$ as $\tau\to 0$. By iteration, it follows that the discrete evolutions $E_\tau(t)$ converge, 
up to a subsequence as $\tau\to 0$, to $E(t)=B_{R(t)}(x(t))$, with 
\[
R(t):= \left( R^3-\frac{3t}{2\pi}\right)^\frac 13 \qquad \text{for }t\in \left[0,\frac 23 \pi R^3\right)
\]
and $x(t)$ is a Lipschitz function such that $|\dot x(t)|\le |\dot R(t)|$ for a.e. $t\in [0,\frac 23 \pi R^3)$.

Notice that in this case the limit evolution is non-unique.

\smallskip

\noindent{\it Example 2.} Let $E^0=B_{R_1}(x_1)\cup B_{R_2}(x_2)$ for some $R_1\ge R_2>0$ and $x_1,x_2\in\R^2$ such that $|x_1-x_2|>R_1+R_2$.
As in the previous example, we have that $T_\tau E^0 = B_{r_1}(\tilde x_1)\cup B_{r_2}(\tilde x_2)$, with $B_{r_i}(\tilde x_i)\subseteq B_{R_i}(x_i)$ for $i\in \{1,2\}$, 
and the radii $r_i$ minimize the function
\begin{equation}\label{funtwoballs}
(r_1,r_2)\mapsto 2\pi (r_1+r_2)  + \frac{\pi^2}{2\tau} \left( R_1^2-r_1^2+ R_2^2-r_2^2\right)^2.
\end{equation}
If $R_1>R_2$, by an easy computation it follows that $r_1=R_1$ and $r_2$ minimize the function in \eqref{funball} with $R$ replaced by $R_2$, 
that is, $r_2 = R_2- \tau/(2\pi R_2^2) + o(\tau)$ as $\tau\to 0$. In particular, in the limit as $\tau \to 0$, we obtain the evolution
$E(t)=B_{R_1}(x_1)\cup B_{R_2(t)}(x_2(t))$, with $R_2(t)$ and $x_2(t)$ as in the previous case of a single ball.

On the other hand, If $R_1=R_2=R$ then either $r_1=R$ and $r_2$ minimize the function in \eqref{funball}, 
or viceversa $r_2=R$ and $r_1$ minimize the function in \eqref{funball}. This implies that, in the limit as $\tau \to 0$,
only one of the two balls start shrinking, whereas the other does not move until the first ball disappears.
As above the limit evolution is non-unique.


\appendix
\section{Convex functions of gradients}\label{appendixconvex}
\subsection{Convex function of measures}
We give here an alternative proof of (a simpler variant of)
the main result of~\cite{DM86}, with less hypotheses on $F$. 
We consider $F:\R^m\to [0,+\infty]$, $m\ge 1$ a convex, lower semicontinuous function.

We start by assuming that $F(0)=0$ so that $0$ is a minimizer of $F$.
In particular, denoting $F^*$ the convex conjugate of $F$, one has
$F^*(q) = \sup_p p\cdot q-F(p)\ge 0$ (choosing $p=0$), and $F^*(0)=-\min_p F(p)=0$,
hence $0$ is also a minimizer of $F^*$.

Under these assumptions, one has the following variant of~\cite[Thm~2.1]{DM86}:
\begin{theorem}\label{th:DM}
  For any vectorial Borel (or Radon) measure $\mu$, $A\subset\Om$ open,
  \begin{multline*}
    \sup\left\{ \int_{\Omega}\varphi\cdot\mu -\int_\Omega F^*(\varphi(x))\,dx:
      \varphi\in C_c^\infty(A;\R^m)\right\}
    \\= \int_A F(\mu^a(x))dx + \int_A F^\infty\left(\frac{\mu^s}{|\mu^s|}(x)\right)|\mu^s|(x)
  \end{multline*}
  (possibly infinite).
\end{theorem}
In this statement, $F^\infty$ is the recession function of $F$, given by
\begin{equation}\label{eq:recession}
  F^\infty(p) = \lim_{t\to +\infty} \frac{1}{t}F(tp+\hat p)=
  \sup_{t>0}\frac{1}{t}F(tp+\hat p) = \sup_{q: F^*(q)<+\infty} q\cdot p
\end{equation}
(where $\hat p$ is any point in the relative interior of the domain of $F$),
see~\cite[\S 1]{DM86}; $\mu=\mu^a(x)dx+\mu^s$ is the Radon-Nikod\'ym decomposition of $\mu$ with
respect to the Lebesgue measure; $\frac{\mu^s}{|\mu^s|}$ is the Radon-Besicovitch derivative
of the singular part $\mu^s$ with respect to its variation $|\mu^s|$.
\begin{proof}
  We define a measure as
  \[
    \lambda(A) = \sup\left\{ \int_{\Omega}\varphi\cdot\mu -\int_\Omega F^*(\varphi(x))\,dx:
      \varphi\in C_c^\infty(A;\R^m)\right\}
  \]
  for $A$ open.

  \paragraph{Step 1.} We claim that
  \begin{itemize}
  \item for any open sets $A,B\subset\Omega$,   $\lambda(A\cup B)\le\lambda(A)+\lambda(B)$,
  \item with equality if $A\cap B=\emptyset$, and   
  \item for any open set $A$, $\lambda(A)=\sup\{\lambda(B): B\subset\Om\text{ open},
    \overline{B}\subset A\}$.
  \end{itemize}
  Then thanks to De~Giorgi-Letta's theorem~\cite[Thm~1.53]{AFP}, the extension
  \[
    \lambda(B) = \inf\{\lambda(A): B\subset A\subset\Om, A\text{ open}\}
  \]
  defines a metric outer measure on $\Om$, and in particular a Borel positive measure.
  As the two last points follow quite obviously from the definition of $\lambda$,
  the only point to check is the first one: we consider $A,B$ open sets, possibly
  intersecting, and given $\eps>0$ we choose $\varphi\in C_c^\infty(A\cup B;\R^m)$ such that
  \[
    \lambda(A\cup B)\le \int_{\Omega}\varphi\cdot\mu -\int_\Omega F^*(\varphi(x))\,dx + \eps.
  \]
  (If $\lambda(A\cup B)=+\infty$, we require rather than the integrals are larger
  than $1/\eps$, the rest of the proof is modified accordingly.)
  We consider a smooth partition of the unity $\eta_A,\eta_B$ subject to the sets $A,B$.
  Then, using that $F^*(\eta_A\varphi)\le \eta_AF^*(\varphi)+ (1-\eta_A)F^*(0)\le F^*(\varphi)$
  (as $0$ is a minimizer of $F^*$), denoting $\varphi_{A}=\varphi\eta_{A}$,
  $\varphi_{B}=\varphi\eta_{B}$, so that in particular $\varphi=\varphi_A+\varphi_B$,
  one has:
  \begin{multline*}
    \int_{\Omega}\varphi\cdot\mu -\int_\Omega F^*(\varphi(x))\,dx
    \\\le\int_{\Omega}\varphi_A \cdot\mu -\int_\Omega F^*(\varphi_A(x))\,dx
    +\int_{\Omega}\varphi_B \cdot\mu -\int_\Omega F^*(\varphi_B(x))\,dx
    \\+ \int_{\{\eta_A\eta_B>0\}} F^*(\varphi(x))\,dx.
  \end{multline*}
  As it is possible to chose $\eta_A,\eta_B$ such that the intersection
  of their support $\{\eta_A\eta_B>0\}$ is arbitrarily small, one can assume
  that $\int_{\{\eta_A\eta_B>0\}} F^*(\varphi(x))\,dx\le\eps$. We deduce that
  \[
    \lambda(A\cup B)\le\lambda(A)+\lambda(B) + 2\eps,
  \]
  and as $\eps$ is arbitrary, the claim follows.

  \paragraph{Step 2.} Quite obviously, for any $\varphi\in C_c^\infty(A;\R^m)$ such
  that $\varphi(x)\in \dom F^*$ a.e.,
  \begin{multline*}
    \int_A F(\mu^a)dx + \int_A F^\infty\left(\frac{\mu^s}{|\mu^s|}\right)|\mu^s|
    \ge \int_A \varphi\cdot\mu^a - F^*(\varphi) dx +
    \int_A \varphi\cdot \frac{\mu^s}{|\mu^s|} |\mu^s|
    \\=
    \int_A \varphi\cdot\mu -\int_A F^*(\varphi) dx,
  \end{multline*}
  so that one needs only to show that in the sense of measures, $\lambda \ge F(\mu^a)dx + F^\infty\left(\frac{\mu^s}{|\mu^s|}\right)|\mu^s|$.

  Let $x\in\Om$, $r>0$ with $B_r(x)\subset\Om$, and choose any $q\in\dom F^*$.
  Let $\eta\in C_c^\infty(B_r(x);[0,1])$ be a smooth cutoff. Then,
  $\lambda (B_r(x))\ge \int_{B_r(x)} \eta q\cdot\mu - \int_{B_r(x)} F^*(\eta q)dz$.
  As before, as $0$ is a minimizer of $F^*$ and $\eta(x)\in [0,1]$, one has
  $F^*(\eta q) \le \eta F^*(q) + (1-\eta)F^*(0)\le F^*(q)$, so that:
  \[
    \lambda (B_r(x))\ge q\cdot \mu(B_r(x)) - |B_r(x)|F^*(q) -
    |q|\int_{B_r(x)} (1-\eta)|\mu|.
  \]
  Now, we may send $\eta$ to $\chi_{B_r(x)}$, and this sends the last term to $0$.
  It follows that
  \begin{equation}\label{eq:lambdaball}
    \lambda(B_r(x)) \ge q\cdot \mu(B_r(x)) - |B_r(x)|F^*(q).
  \end{equation}

  Next, we consider a point $x$ where both the derivatives of $\mu$ and $\lambda$ with respect
  to the Lebesgue measure exist. Then, one has:
  \[
    \mu^a(x) = \lim_{r\to 0} \frac{\mu(B_r(x))}{|B_r(x)|}.
  \]
  We rewrite~\eqref{eq:lambdaball} as
  \[
    \frac{\lambda(B_r(x))}{|B_r(x)|} \ge q\cdot\frac{\mu(B_r(x))}{|B_r(x)|} - F^*(q),
  \]
  and let $r\to 0$. We find that $\lambda^a(x) \ge q\cdot \mu^a(x)-F^*(q)$, and since this
  holds for any $q$ with $F^*(q)<+\infty$,we deduce $\lambda^a(x)\ge F^a(\mu^a(x))$.

  We now consider the singular part $\mu^s$. We recall that the Radon-Nikod\'ym derivation
  theorem (or the more general Besicovitch's derivation theorem~\cite[Thm.~2.22]{AFP})
  states that $\mu^s = \mu\restr E$ where:
  \[
    E = \left\{ x\in\Om: \lim_{r\to 0} \frac{|\mu|(B_r(x))}{|B_r(x)|} = +\infty\right\}.
  \]
  In addition, the same theorem ensures that one can further restrict $E$ to the
  points $x$ where $|\mu^s|(B_r(x))>0$ for all $r>0$ and such that
  \[
    \lim_{r\to 0} \frac{\mu^s(B_r(x))}{|\mu^s|(B_r(x))}=\frac{\mu^s}{|\mu^s|}(x)
  \]
  exists and has norm $1$; then $\mu^s = \frac{\mu^s}{|\mu^s|}|\mu^s|$. We remark also that
  $|\mu|$-a.e.~in $E$, this limit is also:
  \[
    \lim_{r\to 0}\frac{\mu(B_r(x))}{|\mu|(B_r(x))}=\frac{\mu}{|\mu|}(x)
  \]
  since $\mu = \frac{\mu}{|\mu|}|\mu|$ so that $\mu^s=\frac{\mu}{|\mu|}|\mu|\restr E$, which
  may hold only if $|\mu^s|=|\mu|\restr E$ and $\frac{\mu}{|\mu|}=\frac{\mu^s}{|\mu^s|}$
  $|\mu|$-a.e.~in $E$. We further restrict $E$ to the points where this holds.
  
  For $x\in E$,  we then consider:
  \[
    \ell(x):=  \limsup_{r\to 0} \frac{\lambda(B_r(x))}{|\mu|(B_r(x))} \in [0,+\infty],
  \]
  then thanks to~\cite[Prop.~2.21]{AFP}, for any $t\ge 0$ and any
  Borel set $F\subset \{x\in E: \ell(x)>t\}$,
  \[
    \lambda(F)\ge t |\mu|(F).
  \]
  In particular for $F\subset \{x\in E:\ell(x)=+\infty\}$, either $|\mu|(F)=0$, or
  $\lambda(F)=+\infty$.

  Let $x\in E':=\{x\in E:\ell(x)<+\infty\}$.
  Using~\eqref{eq:lambdaball} again, we find that for any $q\in \dom F^*$,
  \[
    \frac{\lambda(B_r(x))}{|\mu|(B_r(x))}  \ge q\cdot \frac{\mu(B_r(x))}{|\mu|(B_r(x))} - F^*(q)\frac{|B_r(x)|}{|\mu|(B_r(x))}.
  \]
  By definition of $E$, $|B_r(x)|/(|\mu|(B_r(x)))\to 0$ as $r\to 0$, hence taking
  the limsup, we find that $\ell(x)\ge q\cdot \frac{\mu^s}{|\mu^s|}(x)$. Since
  this is true for any $q$ with $F^*(q)<+\infty$, it follows $\ell(x)\ge F^\infty\left(\frac{\mu^s}{|\mu^s|}(x)\right)$. We then proceed as in the proof of~\cite[Thm~2.22]{AFP}
  to deduce that $\lambda\restr E'\ge \ell |\mu|\restr E'$ and it follows
  \[
    \lambda \restr E \ge F^\infty\left(\frac{\mu^s}{|\mu^s|}(x)\right) |\mu^s|.
  \]
\end{proof}

\begin{remark}The measure $\lambda$ is usually denoted $F(\mu)$.
\end{remark}

\begin{corollary}
  Let $F:\R^n\to [0,+\infty]$ be convex, lsc, and assume there is $\hat p\in\dom F$
  a minimizer of $F$. Then,
  for any $A$ open and bounded,
  \begin{multline*}
    \sup\left\{ \int_{A}\varphi\cdot\mu -\int_A F^*(\varphi(x))\,dx:
      \varphi\in C_c^\infty(A;\R^m)\right\}
    \\= \int_A F(\mu^a(x))dx + \int_A F^\infty\left(\frac{\mu^s}{|\mu^s|}(x)\right)|\mu^s|(x)
  \end{multline*}
\end{corollary}
\begin{proof}
  Let $\tilde F(p) = F(\hat p+p) - F(\hat p)$. Then $\tilde F$ satisfies the assumptions
  of Theorem~\ref{th:DM}. In addition, observe that
  \[
    \tilde F^*(q) = \sup_p q\cdot p -F(\hat p+p)+F(\hat p) = F(\hat p)-q\cdot\hat p + F^*(q),
  \]
  and $\tilde F^\infty=F^\infty$.
  Hence, for any $\R^n$-valued Radon measure $\mu$,
  \begin{multline*}    
    \sup\left\{ \int_{A}\varphi\cdot\mu -\int_A F^*(\varphi(x))\,dx + \int_A \varphi(x)\cdot \hat p\,dx- |A|F(\hat p):
      \varphi\in C_c^\infty(A;\R^m)\right\}
    \\
    = \int_A F(\hat p+\mu^a(x))dx - |A|F(\hat p)+ \int_A  F^\infty\left(\frac{\mu^s}{|\mu^s|}(x)\right)|\mu^s|(x)
  \end{multline*}
  Writing the above equality for the shifted measure $\mu-\hat p\, dx$ shows the claim.
\end{proof}

\subsection{Convex functions of gradients}
Now, we consider $F:\R^d\to [0,+\infty]$ convex, lsc, with $F(0)=0$,
and in addition, we assume there exist $a>0$, $b\ge 0$
such that
\[
  F(p)\ge a|p|-b
\]
for all $p\in \R^d$.
In particular, $F^*(q)\le b+\delta_{B_a(0)}$ (the characteristic of the ball of radius $a$),
and for $u\in L^1(\Om)$,
\begin{equation}\label{eq:supBV}
  \sup
  \left\{ -\int_{\Omega} u \Div\varphi\,dx -\int_\Omega F^*(\varphi(x))\,dx:
    \varphi\in C_c^\infty(A;\R^m)\right\}
\end{equation}
can be bounded only if $a|Du|(A)<+\infty$, that is if $u\in BV(A)$. (Of course, depending on $F$,
higher integrability on the gradient might also be implied.)
In that case, we define as before the measure $F(Du) := F(D^au)dx + F^\infty(D^su)$, with
$F^\infty(D^s u) = F^\infty(D^su/|D^su|)|D^su|$, and Theorem~\ref{th:DM} yields that the
sup in~\eqref{eq:supBV} is nothing but $F(Du)(A)=\int_A F(Du)$.

We therefore can define a convex, lsc.~functional for $u\in L^1(A)$ (or $L^1_\loc(\Om)$):
\[
  \Phi(u;A) := F(Du)(A)=\int_A F(Du)
\]
when $u\in BV(A)$, and $\Phi(u;A)=+\infty$ else; we denote $\Phi(u)=\Phi(u;\Om)$.
We prove here a series of useful lemmas.

\begin{lemma}\label{lem:submod} For any $u,v\in L^1(\Om)$, 
  \[
    \Phi(u\wedge v)+\Phi(u\vee v) \le \Phi(u)+\Phi(v)
    \eqno \eqref{eq:submod}
  \]
\end{lemma}
\begin{proof}
  Let $\rho$ be a symmetric mollifier with support in the unit ball, and $\rho_\eps(x)=\eps^{-d}\rho(x/\eps)$.
  Let also $\Om_\eps=\{x\in\Om:\dist(x,\partial\Om)>\eps\}$.
  Then,
  \begin{multline*}
    \int_{\Om_\eps} F(D (u*\rho_\eps)) dx
    \\=
    \sup  \left\{ -\int_{\Omega} u*\rho_\eps \Div\varphi\,dx -\int_\Omega F^*(\varphi(x))\,dx:
      \varphi\in C_c^\infty(\Om_\eps;\R^d)\right\}
    \\ = 
    \sup  \left\{ -\int_{\Om} u \Div (\rho_\eps*\varphi)\,dx -\int_\Omega F^*(\varphi(x))\,dx:
      \varphi\in C_c^\infty(\Om_\eps;\R^d)\right\}
    \\
    \le
    \sup  \left\{ -\int_{\Omega} u \Div (\rho_\eps*\varphi)\,dx -\int_\Omega F^*(\rho_\eps*\varphi(x))\,dx:
      \varphi\in C_c^\infty(\Om_\eps;\R^d)\right\} \le \Phi(u),
  \end{multline*}
  where we have used Jensen's inequality:
  \[
    F^*\left(\int_{B_1(0)}\rho(z)\varphi(x-\eps z)\,dz\right)\le
    \int_{B_1(0)}\rho(z)F^*\left(\varphi(x-\eps z)\right)\,dz.
  \]
  On the other hand, for $A\subset\subset\Om$,
  \[
    \Phi(u;A)\le \liminf_{\eps\to 0}\int_{A} F(D (u*\rho_\eps)) dx
    \le \liminf_{\eps\to 0}\int_{\Om_\eps} F(D (u*\rho_\eps)) dx
  \]
  by lower-semicontinuity of $\Phi(\cdot;A)$. Since $\Phi(u)=\sup_{A\subset\subset \Om}\Phi(u;A)$, we
  deduce that
  \[
    \lim_{\eps\to 0}\int_{\Om_\eps} F(D (u*\rho_\eps)) dx= \Phi(u).
  \]
  Given $u,v\in L^1(\Om)$, one has (quite obviously, and even if it is
  not finite):
  \[
    \int_{\Om_\eps}\!\! F(D (u*\rho_\eps \wedge v*\rho_\eps))dx
    +     \int_{\Om_\eps}\!\! F(D (u*\rho_\eps \vee v*\rho_\eps))dx
    = 
    \int_{\Om_\eps}\!\! F(D (u*\rho_\eps ))dx
    +     \int_{\Om_\eps}\!\! F(D (v*\rho_\eps))dx.
  \]
  Passing to the limit, and using the lower-semicontinuity of $\Phi$ again, we deduce~\eqref{eq:submod}.
\end{proof}
\begin{lemma}\label{lem:density}
  Let $D_\infty :=\{ u\in L^1(\Om)\cap \dom \Phi: \partial\Phi(u)\cap L^\infty(\Om)\neq\emptyset\}$.
  Then $\overline{D_\infty} \supseteq \dom\Phi$.
\end{lemma}
\begin{proof}
  Let $u\in \dom\Phi$ and for $k\ge 1$, $u_k=(-k)\vee(u\wedge k)$. Then $u_k\to u$
  (in $L^1(\Om)$, or $L^p(\Om)$ if $u\in L^p(\Om)$, $p\in [1,+\infty]$; in particular using
  $\Phi(u)<+\infty$ this holds for any $p\le d/(d-1)$). We claim that $\Phi(u_k)\le\Phi(u)$,
  in fact, Lemma~\ref{lem:submod} shows that $\Phi(u\wedge k)+\Phi(u\vee k)
  \le \Phi(u)+\Phi(k)=\Phi(u)$, the claim follows.
  We deduce (thanks to the lower semicontinuity
  of $\Phi$) $\lim_{k\to+\infty}\Phi(u_k)= \Phi(u)$.)
  
  Now, we consider $v_{k,l}$ the minimizer of
  \[
    \frac{l}{2}\int_\Om (v-u_k)^2 dx + \Phi(v).
  \]
  We have that $\|v_{k,l}\|_\infty\le k$, indeed otherwise letting
  $v'=(-k)\vee(v_{k,l}\wedge k)$, one would have $\int_\Om (v'-u_k)^2dx< \int_\Om(v_{k,l}-u_k)^2dx$,
  and $\Phi(v')\le\Phi(v_{k,l})$ (as before).
  The Euler-Lagrange equation for this problem can be written:
  \[
    q_{k,l}:=l(u_k-v_{k,l}) \in \partial\Phi(v_{k,l}),
  \]
  and since $\|q_{k,l}\|_\infty\le 2k l$,
  this shows that $\partial\Phi(v_{k,l})\cap L^\infty(\Om)\neq\emptyset$.
  When $l\to\infty$, $v_{k,l}\to u_k$ (in $L^2(\Om)$, and then also $L^1(\Om)$ or
  $L^p(\Om)$ for any $p<+\infty$), which shows the Lemma.
\end{proof}
We state a last result which
describes the subgradient of $\Phi$ in the simpler
case where $F$ and $F^*$ are superlinear, so that, in particular,
the domain of $\Phi$ is a subset of $W^{1,1}(\Om)$.
The case where $F$ is one-homogeneous is
discussed for instance in~\cite{MollAnisotropic},
and has been recently generalized to the Lipschitz
case in~\cite{gorny2022dualitybased}. 
The fully general case is a difficult issue. 
In the lemma below, one considers $\Phi$ as a functional
in $L^p(\Om)$, $1\le p\le+\infty$, and $\Phi^*$ is defined
in $L^{p'}(\Om)$ with $1/p+1/p'=1$. In the non-reflexive
cases $p\in\{1,+\infty\}$, the convergence $-\Div\vp_n$ to $w$
below has to be understood in the weak or weak-$*$ sense.


\begin{lemma}\label{lem:subgrad}
  Assume $F,F^*$ have full domain (equivalently, $F,F^*$ are superlinear).
  If $w\in\partial \Phi(u)$, there exists $z\in L^1(\Om;\R^d)$
  such that $z(x)\in\partial F(Du(x))$ a.e.~and $w=-\Div z$ in
  the distributional sense (with $z\cdot\nu_\Om=0$ in the weak
  sense on $\partial\Om$, that is, $\int z\cdot Dv\,dx=\int w v\,dx$
  for any $w\in W^{1,\infty}(\Om)$). In particular, if $F$ is $C^1$,
  one has $w = -\Div D F(Du)$.
\end{lemma}
\begin{proof} We first give in a first step
  a partial description of the conjugate
  of $\Phi$ (see~\cite{bouchitte2020convex} for a description a general setting); then in Step~2, we characterize the subgradients.

  \paragraph{Step~1 - Description of the conjugate}
  We introduce the convex function:
  \[
    H(w) :=\begin{cases}
      \min \left\{\int_\Om F^*(\vp)\,dx: \vp\in C_c^\infty(\Om;\R^d), -\Div \vp=w\right\}
      & \text{ if  this set is nonempty;}\\
      +\infty& \text{ else.}
    \end{cases}
  \]
  Then one has (Theorem~\ref{th:DM})
  $\Phi(u)=\int_\Om F(Du) = \sup_{q} \int_\Om w u\,dx- H(w) = H^*(u)$.
  Hence, $\Phi^*(w) = H^{**}(w)$ is the lsc.~envelope of $H$ and one has:
  \[
    \Phi^*(w) = \inf\left\{ \liminf_n \int_\Om F^*(\vp_n)\,dx:
      \vp_n\in C_c^\infty(\Om;\R^d),\ 
      -\Div\vp_n\to w \text{ in }L^{p'}(\Om)\right\}.
  \]
  Now, since $F^*$ is superlinear, if $(\vp_n)_n$ is a  minimizing
  sequence in the above infimum, and if the latter is finite,
  $(\vp_n)_n$  has (up to a subsequence)
  a weak limit $z$ in $L^1(\Om;\R^d)$. For any
  $v\in W^{1,\infty}(\Om)$, one has
  \[
    \int_\Om \vp_n\cdot Dv\,dx = -\int_\Om \Div\vp_n v\,dx
  \]
  so that in the limit,
  \begin{equation}\label{eq:divw}
    \forall v\in W^{1,\infty}(\Om)\,, \quad
    \int_\Om z \cdot Dv\,dx = \int_\Om  w v\,dx.
  \end{equation}
  (In particular, of course, $-\Div z=w$ in the sense of distributions.)
  
  In addition, by lower-semicontinuity, one has
  \begin{equation*}
    \int_\Om F^*(z)\,dx \le \liminf_n\int_\Om F^*(\vp_n)\,dx
    = \Phi^*(w).
  \end{equation*}
  
  \paragraph{Step 2 - Subgradient}
  Let now  $w\in\partial\Phi(u)$. Equivalently $u\in\partial\Phi^*(w)$,
  and one has
  \[
    \Phi(u) + \Phi^*(w) = \int_\Om u w\,dx
  \]
  Since $\Phi^*(w)<+\infty$, we may consider $(\vp_n)_n$ and $z$ as built in Step~1,
  with $\vp_n\wto z$ in $L^1(\Om)$.

  One has $u\in W^{1,1}(\Om)$ and $A_n:=F(Du)+F^*(\vp_n)-\vp_n\cdot Du\ge 0$ a.e.,
  while
  \[
    \int_\Om A_n \,dx = \Phi(u)+\int_\Om F^*(\vp_n)\,dx +\int_\Om u\Div\vp_n\,dx
    \to \Phi(u)+\Phi^*(w)-\int_\Om w u\,dx = 0
  \]
  since $-\Div\vp_n\to w$ in $L^{p'}(\Om)$ and $u\in L^p(\Om)$. Hence,
  $A_n\to 0$ in $L^1(\Om)$. In particular, for any $M\ge 0$,
  letting $E_M:=\{x:|Du(x)|\le M\}$, one has
  $\int_{E_M} A_n\,dx \to 0$, $\int_\Om \vp_n\cdot (\chi_{E_M}Du)\,dx\to \int_{E_M} z\cdot Du\,dx$, and $\int_{E_M} F^*(z)\,dx\le\liminf_n\int_{E_M} F^*(\vp_n)$ and
  we deduce
  \[
    \int_{E_M} F^*(z)+F(Du) - z\cdot Du\,dx \le 0.
  \]
  Since $M$ is arbitrary,
  it follows that $F(Du)+F^*(z) = z\cdot Du$ (equivalently $z\in\partial F(Du)$) a.e.~in~$\Om$.
\end{proof}

\subsection{The Dirichlet case}\label{app:Diri}

We describe here how the previous results should be adapted to consider Dirichlet
boundary constraints. In this section, $\Om$ is a bounded, open, Lipschitz-regular
set in $\R^d$. We consider also $u^0\in L^1(\partial\Om)$ and
we assume that
there exists an extension $u^0\in W^{1,1}(\Om)$ and $t>1$ with
\begin{equation}\label{eq:condu0}
  \int_\Om F(tDu^0(x))\,dx < +\infty.
\end{equation}
(If $F$ is Lipschitz this is  a standard result, if $F$ has a higher growth
this imposes an additional constraint on $u^0_{|\partial\Om}$.)

Let us introduce the functional:
\[
  \Psi_0(u) = \begin{cases}
    \Phi(u) \in [0,+\infty] & \text{ if } u\in W^{1,1}(\Om), u=u^0 \text{ on }\partial\Om\,,\\
    +\infty & \text{ else.}
  \end{cases}
\]
In case $F$ is superlinear, it is well known that $\Psi=\Psi_0$ is
lower-semicontinuous, and can be recovered by duality as:
\begin{equation}\label{eq:defPsi}
  \Psi(u) = \sup\left\{ \int_{\partial\Om} u^0\vp\cdot\nu_\Om d\H^{d-1} - 
    \int_\Om (u\Div\vp + F^*(\vp)) dx : \vp\in C^\infty(\overline{\Om};\R^d)\right\}
\end{equation}
(where $\nu_\Om$ is the outer normal to $\partial\Om$)
if $u=u^0$ on $\partial\Om$. Here we have simply used Green's formula and the boundary condition to transform the term $\int_\Om\vp\cdot Du \,dx$. 

In the general case, one checks that the functional $\Psi$ given by~\eqref{eq:defPsi} is
also given, for a general $u\in L^1(\Om)$ by:
\begin{equation}\label{eq:Psi2}
  \Psi(u) =\begin{cases} \int_\Om F(Du) + \int_{\partial\Om} F^\infty((u^0-u)\nu_u) d\H^{d-1} \in [0,+\infty] & \text{ if } u\in BV(\Om)\,,\\
    +\infty & \text{ else.}
  \end{cases}
\end{equation}
This is easily obtained applying Theorem~\ref{th:DM} to the measure $Du\restr\Om + (u-u^0)\nu_\Om\H^{d-1}\restr\partial\Om$ (in a larger domain), when $u$ is $BV$, so that
in particular its trace is well defined on the boundary.

The fact $\Psi$ is then the lower-semicontinuous envelope of $\Psi_0$ is well
known in case $F$ is superlinear (since then $\Psi=\Psi_0$) and in
case $F$ is Lipschitz (approximating $u$ with
$u + \chi_{\{\dist(\cdot,\partial\Om)<1/n\}}(u^0-u)$). The situation where
$F$ is neither Lipschitz nor superlinear remains unclear, we choose
in that case to use~\eqref{eq:defPsi}--\eqref{eq:Psi2} as a definition
for the Dirichlet problem.

We turn ourself to the characterization of the subgradient of $\Psi$.
As before, we restrict ourselves to the Lipschitz case, for
which we refer to~\cite{gorny2022dualitybased}:
in this case, one can show that $w\in\partial\Psi(u)$ if and only
if $w=-\Div z$ with $z\cdot Du=F(Du)+F^*(z)$ in the sense of~\cite{Anzellotti},
and $z\cdot\nu_\Om= F(\sign(u-u^0)\nu_\Om)$ a.e.~on the boundary.
In the superlinear case, the following holds:
\begin{lemma}\label{lem:subgradDiri}
  Assume $F,F^*$ have full domain (equivalently, $F,F^*$ are superlinear).
  If $w\in\partial \Psi(u)$, there exists $z\in L^1(\Om;\R^d)$
  such that $z(x)\in\partial F(Du(x))$ a.e.~and $w=-\Div z$ in
  the distributional sense.
  In particular, if $F$ is $C^1$,   one has $w = -\Div D F(Du)$.
\end{lemma}
\begin{proof}
  The proof is almost the same as for Lemma~\ref{lem:subgrad} and we
  only sketch it. We introduce again:
  \[
    H(w) :=\begin{cases}
      \min \left\{\int_\Om F^*(\vp)dx -\int_{\partial\Om} u^0\vp\cdot\nu_\Om
        \,d\H^{d-1}: \vp\in C_c^\infty(\overline{\Om};\R^d), -\Div \vp=w\right\}
      \\ \hfill \text{ if  this set is nonempty;} \\
      +\infty \text{ else.}
    \end{cases}
  \]
  and we find that since $\Psi=H^*$, $\Psi^*$ is the convex, l.s.c.~envelope of $H$. Hence if $w\in\partial\Psi(u)$ (where $u=u^0$ on $\partial\Om$,
  since $F^\infty\equiv+\infty$), there is $\vp_n$ with
  $-\Div\vp_n\to w$,
  \[
    \lim_n \int_\Om F^*(\vp_n)dx -\int_{\partial\Om} u^0\vp_n\cdot\nu_\Om
    \,d\H^{d-1} = \Psi^*(w)
  \]
  and
  \[ \Psi^*(w)+\Psi(u) = \int_\Om w u\,dx. \]
  Using
  \[
    \int_{\partial\Om } u^0\vp_n\cdot \nu_\Om \,\H^{d-1}
    = \int_\Om u^0\Div\vp_n + \vp_n\cdot Du^0\,dx
    \le C+\frac{1}{t}\int_\Om F(tDu^0) + F^*(\vp_n)dx,
  \]
  where $t>1$ is from \eqref{eq:condu0}, we deduce that
  $\int F^*(\vp_n)dx$ is bounded and as in the Neumann case
  we may assume, up to a subsequence, that $\vp_n$ converges
  weakly to some vector field $z\in L^1(\Om;\R^n)$ with $-\Div z=w$.
  We let again $A_n= F(Du) + F^*(\vp_n) - \vp_n\cdot Du\ge 0$, and use
   Green's formula to obtain:
  \begin{multline*}
    \int_\Om A_n\,dx 
     = \int_\Om F(Du) + F^*(\vp_n)dx
    + \int_\Om u\Div\vp_n\,dx - \int_{\partial\Om} u_0\vp_n\cdot\nu_\Om\,d\H^{d-1}
    \\= \Psi(u) + \int_\Om F^*(\vp_n)dx
    - \int_{\partial\Om} u_0\vp_n\cdot\nu_\Om\,d\H^{d-1}
    + \int_\Om u\Div\vp_n\,dx\stackrel{n\to\infty}\longrightarrow 0.
  \end{multline*}
  We conclude as in the proof of Lemma~\ref{lem:subgrad}.
\end{proof}



\end{document}